\documentclass[reqno,a4paper]{amsart}

\usepackage[latin1,utf8]{inputenc}
\usepackage[english]{babel}
\usepackage{eucal,amsfonts,amssymb,amsmath,amsthm,epsfig,mathrsfs}
\usepackage{cancel,soul}
\usepackage{color}
\usepackage{amscd,amsxtra}
\usepackage{enumerate}
\usepackage{latexsym}
\usepackage{bm}
\usepackage{setspace}
\usepackage{textcomp}

\allowdisplaybreaks

\newcounter{ipotesi}

\Alph{ipotesi}
 \makeatletter \@addtoreset{equation}{section}

\makeatother \makeatletter

\newtheorem{thm}{Theorem}[section]
\newtheorem{hyp}[thm]{Hypotheses}{\rm}
{\rm}
\newtheorem{lemm}[thm]{Lemma}
\newtheorem{cor}[thm]{Corollary}

\newtheorem{prop}[thm]{Proposition}
\newtheorem{defi}[thm]{Definition}
\newtheorem{rmk}[thm]{Remark}{\rm}

\newcounter{parentenv}

\newcommand{\R}{{\mathbb R}}

\newcommand{\E}{{\mathbb E}}
\newcommand{\N}{{\mathbb N}}

\newcommand{\X}{{\mathcal{X}}}
\newcommand{\K}{{\mathcal{K}}}
\newcommand{\D}{{\mathcal{D}}}
\newcommand{\J}{{\mathcal{D}}}

\newcommand{\eps}{\varepsilon}
\newcommand{\ra}{\rightarrow}

\newcommand{\ol}[1]{\overline{#1}}

\renewcommand{\tilde}[1]{\widetilde{#1}}

\newcommand{\Dom}{{\operatorname{Dom}}}

\newcommand{\Id}{{\operatorname{I}}}

\newcommand{\set}[1]{{\left\{#1\right\}}}
\newcommand{\pa}[1]{{\left(#1\right)}}
\newcommand{\sq}[1]{{\left[#1\right]}}

\newcommand{\abs}[1]{{\left|#1\right|}}
\newcommand{\norm}[1]{{\left\|#1\right\|}}
\newcommand{\scal}[2]{{\left\langle #1,#2\right\rangle}}

\newcommand{\mi}[1]{{\lbrace #1(t,x)\rbrace_{t\geq 0}}}

\newcommand{\eqsys}[1]{{\left\{\begin{array}{ll}#1\end{array}\right.}}
\newcommand{\tc}{\, \middle |\,}

\makeatletter
\@namedef{subjclassname@2020}{%
  \textup{2020} Mathematics Subject Classification}
\makeatother

\begin{document}

\frenchspacing


\title[Schauder regularity results in separable Hilbert spaces]{Schauder regularity results in separable Hilbert spaces}

\author[D.A. Bignamini and S. Ferrari ]{Davide A. Bignamini, Simone Ferrari$^*$}

\address{D.A.B.: Dipartimento di Matematica, Universit\`a degli Studi di Pavia, Via Adolfo Ferrata 5, 27100 PAVIA, Italy}
\email{\textcolor[rgb]{0.00,0.00,0.84}{davideaugusto.bignamini@unipv.it}}

\address{S.F.:  Dipartimento di Matematica e Fisica ``Ennio De Giorgi'', Universit\`a del Salento, Via per Arnesano, 73100 LECCE, Italy}
\email{\textcolor[rgb]{0.00,0.00,0.84}{simone.ferrari@unisalento.it}}

\thanks{$^*$Corresponding author.}



\date{\today}

\begin{abstract}
We prove Schauder type estimates for solutions of stationary and evolution equations driven by weak generators of transition semigroups associated to a semilinear stochastic partial differential equations with values in a separable Hilbert space.
\end{abstract}

\maketitle

\section{Introduction}

Let $\X$ be a real separable Hilbert space with inner product $\scal{\cdot}{\cdot}$ and associated norm $\norm{\cdot}$. Let $\{W(t)\}_{t\geq 0}$ be a $\X$-cylindrical Wiener process defined on a normal filtered probability space $(\Omega,\mathcal{F},\{\mathcal{F}_t\}_{t\geq 0},\mathbb{P})$.
Consider the stochastic partial differential equation
\begin{gather}\label{eqF02}
\eqsys{
dX(t,x)=\big[AX(t,x)+RG(X(t,x))\big]dt+RdW(t), & t>0;\\
X(0,x)=x\in \X,
}
\end{gather} 
where $A:{\rm Dom}(A)\subseteq\X\ra\X$ is a linear (possibly) unbounded operator, $R:\X\ra\X$ is a linear and continuous operator and $G:\X\ra\X$ is a smooth enough function, and consider its mild solution $\{X(t,x)\}_{t\geq 0}$, which exists under suitable conditions. Semilinear stochastic partial differential equations are widely studied in the literature, see for example \cite{AD-BA-MA1, AD-MA-PR1,BON-FUR1,FUR1,MAS1,MAS2}. Under suitable assumptions \eqref{eqF02} has a unique mild solution $\mi{X}$ (see Definition \ref{DEFImild}) and the associated transition semigroup $\{P(t)\}_{t\geq 0}$ defined as
\begin{align}\label{intr_trans}
P(t)f(x):=\E[f(X(t,x))],\qquad t>0,\ x\in\X,
\end{align}
where $f:\X\ra\R$ is a Borel measurable function, is well defined. Moreover $\{P(t)\}_{t\geq 0}$ is a weakly continuous  semigroup (in the sense of \cite[Appendix B]{CER1}) in ${\rm BUC}(\X)$, the space of real-valued, bounded and uniformly continuous functions. Its weak generator $N:{\rm Dom}(N)\subseteq {\rm BUC}(\X)\ra{\rm BUC}(\X)$ is the unique closed operator whose resolvent is given by
\begin{align*}
R(\lambda,N)f(x)=\int_0^{+\infty} e^{-\lambda s}P(s)f(x)ds,\qquad \lambda>0,\ x\in\X,\ f\in{\rm BUC}(\X).
\end{align*}
We recall that the operator $N$ acts on bounded, real-valued, cylindrical and smooth enough functions $\varphi$ as
\[N\varphi(x)=\frac{1}{2}{\rm Trace}[R^2\D^2\varphi(x)]+\langle x,A^*\D \varphi(x)\rangle+\langle R\D G(x),\D \varphi(x)\rangle,\qquad x\in\X,\]
see, for example \cite{GK01}. In this paper we investigate Schauder type estimates for the function
\begin{align}\label{GS}
u(x):=R(\lambda,N)f(x);\qquad \lambda>0,\ x\in\X,\ f\in{\rm BUC}(\X),
\end{align}
namely for the solution of the stationary equation
\begin{align}\label{intr_problema_stazionario}
\lambda u-Nu=f,\qquad \lambda>0,\ f\in{\rm BUC}(\X).
\end{align}
Moreover we will also study Schauder regularity results for the function 
\begin{align}\label{mild_evol}
v(t,x):= P(t)f(x)+\int_0^t P(t-s)g(s,\cdot)(x)ds,\qquad T>0,\ t\in[0,T],\ x\in\X,
\end{align}
where $f,g$ belong to suitable H\"older space. The function $v$ is the mild solution of the evolution equation driven by $N$, namely for $T>0$
\begin{align}\label{evol_prob}
\eqsys{\frac{d}{dt}v(t,x)=Nv(t,x)+g(t,x), & t\in(0,T],\ x\in\X;\\
v(0,x)=f(x), & x\in\X.}
\end{align}
We recall that, even in finite dimension, if $N$ is an operator with unbounded coefficients, then the function $v$ in \eqref{mild_evol} do not gain any regularity with respect to the time variable (see, for example, \cite{Lor17}). So we will prove a Schauder regularity result only for the spatial variable (historically, in finite dimension, such type of results were first obtained in \cite{KCL75,KCL80}, for operator with bounded coefficients, while the first result for the case of unbounded coefficients can be found in \cite{Lun97}). The use of H\"older continuous functions reveals to be useful in several problems scuh as uniqueness in law, pathwise uniqueness and uniqueness of the martingale problem for some stochastic partial differential equations. An important tool to prove the above mentioned Schauder type estimates will be a Bismut--Elworthy--Li type formula for \eqref{intr_trans} along the directions of $R(\X)$ (see Theorem \ref{BEL}).

In the finite dimensional setting Schauder estimates are widely studied, see for example \cite{GT01,KCL75,KCL80,LSU68,LU68}, for the bounded coefficients case and \cite{DL95,Lor17} for the case of unbounded coefficients, while the theory for the infinite dimensional case is less developed. In \cite{CDP96-2} and \cite[Section 6.4.1]{DPZ02} the authors study the case $G\equiv0$ and $R={\rm Id}_{\X}$. In this context, the authors benefit from these three elements: the Mehler representation of the semigroup $\{P(t)\}_{t\geq 0}$, the fact that for every $t>0$ it holds $e^{tA}(\X)\subseteq Q_t^{1/2}(\X)$ and
\begin{align}\label{doom}
\|Q_t^{-1/2}e^{tA}x\|\leq K t^{-1/2}\|x\|,\qquad t>0,\ x\in\X,
\end{align}
for some $K>0$, where $Q_t:=\int_0^te^{2sA}ds$, for any $t>0$, and an interpolation result (see \cite{CDP96-2}). Using these facts the authors of \cite{CDP96-2} and \cite[Section 6.4.1]{DPZ02} prove a optimal Schauder regularity result, namely if $f$ is a $\alpha$-H\"older continuous function with $\alpha\in(0,1)$, then the function $u$ given by \eqref{GS} is twice Fr\'echet differentiable with $\alpha$-H\"older continuous second order derivatives.
In \cite{DA5} the author considers the same framework in the case that $G$ is not identically zero. Compared to the previous situation, one cannot exploit the Mehler representation of the semigroup. In \cite{CL19} the case of the classical Ornstein--Uhlenbeck semigroup is considered, i.e. $A=-(1/2){\rm Id}_{\X}$, $G\equiv 0$ and $R=Q^{1/2}$ where $Q$ is a trace class operator. In this case \eqref{doom} is not satisfied, and in fact the authors of \cite{CL19} prove Schauder type theorems replacing the standard notions of Fr\'echet differentiability and H\"olderianity with a differentiability and H\"olderianity along $Q^{1/2}(\X)$, which is the Cameron--Martin space of the Gaussian measure with mean zero and covariance operator $Q$. In general if \eqref{doom} is not verified, it is not possible to obtain a optimal Schauder regularity result with respect to the standard H\"olderianity and Fr\'echet differentiability (see, for instance, \cite{LR21}), barring a few specific cases (see, for example, \cite[Section 5.5]{CER1}).


The aim of this paper is to unify and extend the results of \cite{CDP96-2,CL19,DA5} and of \cite[Section 6.4.1]{DPZ02}. To do so we will prove a Bismut--Elworthy--Li type formula for \eqref{intr_trans} with respect to a suitable differentiability notion (Theorem \ref{BEL}). Accordingly we will study a notion of Gateaux differentiability for the mild solution of \eqref{eqF02} (Definition \ref{defn_Gateaux}) and we will prove a chain rule associated to this type of differentiability (Corollary \ref{accatena}), which will be fundamental in the proof of the Bismut--Elworthy--Li formula. Using this formula some sharp estimates for the derivatives of the transition semigroup $\{P(t)\}_{t\geq 0}$ will be proved (Propositions \ref{boss}, \ref{bossD} and \ref{Holder}). Another result we will prove in order to obtain the Schauder type estimates of this paper is an interpolation theorem (Theorem \ref{intDS}), similar to the one in \cite{CDP96-2}, which is interesting on its own. The previously mentioned results allow us to prove the optimal Schauder regularity for the solutions $u$ and $v$ of \eqref{intr_problema_stazionario} and \eqref{evol_prob}, respectively (see Theorems \ref{THM_SCHAUDER}, \ref{Thm_Zyg} and \ref{Thm_evol}) 
Moreover, exploiting Theorem \ref{THM_SCHAUDER}, we will prove a Schauder type estimates for the solution of a stationary equation driven by an operator which is not necessarily the weak generator of a transition semigroup (Proposition \ref{SCHvar}).

We remark that the results of this paper can be applied to stochastic partial differential equations such as
\begin{align}\label{cremamani}
\eqsys{
dX(t,x)=\big[AX(t,x)+(-A)^{-\gamma}G(X(t,x))\big]dt+(-A)^{-\gamma}dW(t), & t>0;\\
X(0,x)=x\in L^2([0,1]^d),
}
\end{align}
where $\gamma\geq 0$, $A$ is a realization of the Laplacian operator in $L^2([0,1]^d)$ with appropriate boundary conditions, and $d\in\N$. This example is not covered by the theory developed in \cite{CDP96-2,CL19,DA5}. We stress that \eqref{cremamani} is widely studied in other contexts (see, for example, \cite{AD-BA-MA1,MAS2,MAS1,Mas08} and \cite[Chapter 4]{DA04}). In Section \ref{Sect_Examples} the SPDE \eqref{cremamani} is studied within a more general example, for which \eqref{doom} is not verified unless $\gamma=0$.

We wish to point out that in \cite{LR21} the linear non-local case is considered; \cite{CDP96,DPZ02,PRI1,PZ00} the Gross Laplacian and some of its perturbations are considered; \cite{CDP12} the case of reaction-diffusion equations are studied and \cite{CL21} the non-autonomous linear case is investigated. For other related results see also \cite{ABGP06,ABP05,Zam99}.

The paper is organized in the following way: in Section \ref{Notations} we introduce our hypotheses and state the main results of this paper. In Section \ref{Interpolation} we show an interpolation result which is interesting on its own and crucial for the proof of our results. Section \ref{Regularity_Mild} is dedicated to the proof of some estimates for the derivatives of the mild solution of \eqref{eqF02} which will be used in Section \ref{Regularity_Semigroup} to estimate the derivatives of $P(t)f$ given by \eqref{intr_trans}. In Section \ref{Main_Results} we provide the proofs of our main results (Theorems \ref{THM_SCHAUDER}, \ref{Thm_Zyg} and \ref{Thm_evol}). In Section \ref{Sect_Examples} we provide a large class of operators $A$ and $R$ satisfying our assumptions. We conclude with Appendix \ref{appen} where we give the proof of a result about uniformly continuous functions we use throughout the paper.

{\footnotesize \subsection*{Notations} Let $(\Omega,\mathcal{F},\{\mathcal{F}_t\}_{t\geq 0},\mathbb{P})$ be a filtered probability space. We say that $\{\mathcal{F}_t\}_{t\geq 0}$ is a normal filtration if
\[\mathcal{F}_t=\bigcap_{s>t}\mathcal{F}_s \qquad t\geq 0;\]
and $\mathcal{F}_0$ contains all the elements $A\in\mathcal{F}$ such that $\mathbb{P}(A)=0$. Let $\K$ be a separable Banach space and let $\xi:\Omega\ra \K$ be a random variable (with respect to the $\sigma$-algebra $\mathcal{F}$ and the $\sigma$-algebra of the Borel measurable subsets of $\K$). We denote by 
\[
\mathbb{E}[\xi]:=\int_\Omega \xi(w)\mathbb{P}(d\omega)=\int_\K x[\mathbb{P}\circ\xi^{-1}](dx),
\]
the expectation of $\xi$ with respect to $\mathbb{P}$. In this paper when we refer to a $\K$-valued process we mean an adapted process defined on $(\Omega,\mathcal{F},\{\mathcal{F}_t\}_{t\geq 0},\mathbb{P})$ with values in $\K$. 
We say that a $\K$-valued process $\{Y(t)\}_{t\geq 0}$ is continuous if the map $(Y(\cdot))(\omega):[0,+\infty)\ra \K$ is continuous for $\mathbb{P}$-a.e. $\omega\in\Omega$. We refer to \cite{Oks03} and \cite{Wil91} for notations and basic results reguarding stochastic processes.

Let $\K_1$ and $\K_2$ be two real Banach spaces equipped with the norms $\norm{\cdot}_{\K_1}$ and $\norm{\cdot}_{\K_2}$, respectively. We denote 
by $B_b(\K_1;\K_2)$ the set of the bounded and Borel measurable functions from $\K_1$ to $\K_2$. If $\K_2=\R$, then we simply write $B_b(\K_1)$. We denote by $C_b(\K_1;\K_2)$ (${\rm BUC}(\K_1;\K_2)$, respectively) the space of bounded and continuous (uniformly continuous, respectively) functions from $\K_1$ to $\K_2$. We consider $C_b(\K_1;\K_2)$ and ${\rm BUC}(\K_1;\K_2)$ with the norm
\[
\norm{f}_{\infty}=\sup_{x\in \K_1}\|f(x)\|_{\K_2}.
\]
If $\K_2=\R$ we simply write $C_b(\K_1)$ and ${\rm BUC}(\K_1)$, respectively.
Let $k\in\N$ and let $f:\K_1\ra \K_2$ be a $k$-times Fr\'echet (Gateaux, respectively) differentiable function we denote by $\J^k f(x)$ ($\J_G^k f(x)$, respectively) its Fr\'echet (Gateaux, respectively) derivative of order $k$ at $x\in\X$.

For $k\in\N$ we denote by $\mathcal{L}^{(k)}(\K_1;\K_2)$ the space of continuous multilinear maps from $\K_1^k$ to $\K_2$, if $k=1$ we simply write $\mathcal{L}(\K_1;\K_2)$, while if $\K_1=\K_2$ we write $\mathcal{L}^{(k)}(\K_1)$.

Let $B:\Dom(B)\subseteq \K_1\ra \K_1$ be a linear operator and let $E$ be a subspace of $\K_1$. The part of $B$ in $E$, denoted by $B_E$, is defined as 
\[
B_{E}x:=Bx,\qquad x\in \Dom(B_{E}):=\set{x\in \Dom(B)\cap E\tc Bx\in E}.
\]
Let $H$ be a Hilbert space equipped with the inner product $\scal{\cdot}{\cdot}_{H}$. We say that $Q\in\mathcal{L}(H)$ is \emph{non-negative} (\emph{positive}) if for every $x\in H\setminus\set{0}$
\[
\langle Qx,x\rangle_H\geq 0\ (>0).
\]
On the other hand, $Q \in \mathcal{L}(H)$ is a \emph{non-positive} (resp. \emph{negative}) operator, if $-Q$ is non-negative (resp. positive). Let $Q\in\mathcal{L}(H)$ be a non-negative and self-adjoint operator. We say that $Q$ is a trace class operator if
\begin{align}\label{trace_defn}
{\rm Trace}[Q]:=\sum_{n=1}^{+\infty}\langle Qe_n,e_n\rangle_H<+\infty,
\end{align}
for some (and hence for all) orthonormal basis $\{e_n\}_{n\in\N}$ of $H$. We recall that the trace operator, defined in \eqref{trace_defn}, is independent of the choice of the basis. We refer to \cite{FHHMZ11} and \cite{RE-SI1} for notations and basic results about linear operators and Banach spaces. All the integrals appearing in the paper are integrals in the sense of Bochner (see \cite{DU77}).}

\section{Hypotheses and main results}\label{Notations}

We start by introducing the hypotheses we will assume throughout the paper.
\begin{hyp}\label{STR1}
We assume that the following conditions hold true.
\begin{enumerate}[\rm(i)]
\item\label{STR1.1} $R\in\mathcal{L}(\X)$ is a positive and self-adjoint operator.
\item $A$ is the infinitesimal generator of a strongly continuous semigroup $\{e^{tA}\}_{t\geq0}$ on $\X$.

\item\label{STR1.2} There exists $\eta\in (0,1)$ such that for any $T>0$
\[
\int^T_0 t^{-\eta}\,{\rm Trace}[e^{tA}R^2e^{tA^*}]dt<+\infty.
\]

\item\label{STR1.3} $G:\X\ra\X$ is a continuous, three times Fr\'echet differentiable map, with uniformly continuous derivatives, and there exists a positive constant $M$ such that for any $x\in\X$ and $h_1,h_2,h_3\in\X$
\begin{align*}
\|\J G(x)h_1\|&\leq M\norm{h_1};\\
\|\J^2 G(x)(h_1,h_2)\|&\leq M\norm{h_1}\norm{h_2};\\
\|\J^3 G(x)(h_1,h_2,h_3)\|&\leq M\norm{h_1}\norm{h_2}\norm{h_3}.
\end{align*}
\end{enumerate}
\end{hyp}
\noindent By \cite[Theorem 7.5]{DA-ZA4}, Hypotheses \ref{STR1} ensure that the mild solution $\{X(t,x)\}_{t\geq 0}$ of \eqref{eqF02} exists, is unique and has $\mathbb{P}$-a.e. continuous trajectories. 
\begin{defi}\label{DEFImild}
For any $x\in\X$ we call mild solution of \eqref{eqF02} a process $\{X(t,x)\}_{t\geq 0}$ such that for any $t>0$ and $x\in\X$
\begin{align}\label{sol_mild}
X(t,x)=e^{tA}x+\int_0^t e^{(t-s)A}RG(X(s,x))ds+\int_0^t e^{(t-s)A}RdW(s).
\end{align} 
\end{defi}
We want to point out that Hypothesis \ref{STR1}\eqref{STR1.3} is stronger than the conditions required in \cite[Theorem 7.5]{DA-ZA4}, but in order to obtain our results we will need it.

\begin{prop}\label{accar}
Assume Hypotheses \ref{STR1}\eqref{STR1.1} hold true. The space $H_R:=R(\X)$ is a separable Hilbert space if endowed with the inner product
\begin{align}\label{inn-prod-hr}
\scal{x}{y}_R:=\langle R^{-1}x,R^{-1}y\rangle,\qquad x,y\in H_R.
\end{align}
The associated norm to \eqref{inn-prod-hr} is $\norm{x}_R:=\|R^{-1}x\|$. Furthermore $H_R$ is a Borel measurable space, continuously embedded in $\X$ and, for any $h\in H_R$, it holds $\|h\|\leq \|R\|_{\mathcal{L}(\X)}\|h\|_{H_R}$.
\end{prop}
\noindent We refer to \cite[Theorem 15.1]{KE1} and \cite[Appendix C]{LI-RO1}, for the proof of Proposition \ref{accar} and some basic properties of the space $H_R$. We now state some additional hypotheses we will use throughout the paper.

\begin{hyp}\label{STR2}
Assume that Hypotheses \ref{STR1} hold true, that the part of $A$ in $H_R$, denoted by $A_R$, generates a strongly continuous semigroup on $H_R$ and that there exists $\zeta_R\in\R$ such that, for every $x\in\X$ and $h\in \Dom(A_R)$ 
\begin{equation}\label{dissR}
\scal{[A_R+R\J G(x)]h}{h}_R\leq \zeta_R\norm{h}^2_{R}.
\end{equation} 
\end{hyp}
\noindent It is easy to see that if there exists $w_R \in\R$ such that, for any $x\in\Dom(A_R)$, it holds
\[
\langle A_Rx,x\rangle_R\leq w_R\norm{x}_R^2,
\]
then \eqref{dissR} is verified with $\zeta_R=w_R+\|R\|_{\mathcal{L}(\X)}M$, where $M$ is the constant introduced in Hypothesis \ref{STR1}\eqref{STR1.3}. We stress that in \cite{CDP96-2,CL19,DA5} and \cite[Section 6.4.1]{DPZ02}, $\zeta_R$ is negative, while in this paper we do not assume any sign on $\zeta_R$. 

\begin{rmk}
We want to point out that in general the stochastic convolution
\[\left\{\int_0^t e^{(t-s)A}RdW(s)\right\}_{t\geq 0},\]
is a $H_R$-valued continuous process. Hence it is not possible to solve \eqref{eqF02} in $H_R$.
\end{rmk}

Now we define the functional spaces which will play a main role in this paper. The following notion of differentiability first appeared in \cite{GRO1} and \cite{KUO1}.

\begin{defi}\label{intr_defn_diff} 
Assume Hypothesis \ref{STR1}\eqref{STR1.1} holds true. We say that a function $\Phi:\X\ra \R$ is $H_R$-differentiable at $x\in\X$, if there exists  $L_x\in\mathcal{L}(H_R;\R)$ such that 
\begin{align*}
\lim_{\norm{h}_R\ra 0}\frac{|\Phi(x+h)-\Phi(x)-L_xh|}{\norm{h}_R}=0.
\end{align*}
If it exists, $L_x$ is unique and we set $\J_R\Phi(x):=L_x$.  We say that $\Phi$ is twice $H_R$-differentiable at $x\in\X$ if the map $\J_{R} \Phi:\X\ra \mathcal{L}(H_R;\R)$ is $H_R$-differentiable at the point $x\in \X$, namely there exists a unique  $B_x\in\mathcal{L}(H_R;\mathcal{L}(H_R;\R))$ such that 
\begin{align*}
\lim_{\norm{h}_R\ra 0}\frac{1}{\norm{h}_R}\|\J_R \Phi(x+h)-\J_R\Phi(x)-B_xh\|_{\mathcal{L}(H_R;\R)}=0.
\end{align*}
We call second order $H_R$-derivative of $\Phi$ at the point $x\in \X$ the unique symmetric continuous bilinear form $\J^2 \Phi(x):H_R\times H_R\ra \R$ defined by
\[
\J^2_{R} \Phi(x)(h,k):=(B_xh)k,\qquad x\in\X,\ h,k\in H_R.
\]
For any $k> 2$, we define in the same way a $k$-times $H_R$-differentiable functions $\Phi$ and we denote by $\D^k_R\Phi(x)$ its $H_R$-derivative of order $k$. Observe that $\J^k_{R}\Phi(x)$ beglongs to $\mathcal{L}^{(k)}(H_R;\R)$. Since for $R={\rm Id}_{\X}$, the $H_R$ differentiability coincide with the standard Fr\'echet differentiability, in this case we will drop the subscript $R$.
\end{defi}

\begin{rmk}
Let $f:\X\ra\R$ be a $H_R$-differentiable function. Since $H_R$ is a Hilbert space, by the Riesz representation theorem, for any $x\in\X$ there exists a unique $l_x\in H_R$ such that
\[
\J_R f(x)h=\langle l_x,h\rangle_R,\qquad h\in H_R.
\]
We call $l_x$ the $H_R$-gradient of $f$ at $x\in\X$ and we denote it by $\nabla_R f(x)$.
\end{rmk}

\noindent This notion of $H_R$-differentiability is a sort of Fr\'echet differentiability along the directions of $H_R$ and it was already considered in various papers (see, for example, \cite{ABF21,BF20,BF22,CL19,CL21}). The proof of the following result follows the same arguments used in the proof of \cite[Proposition 17]{BF20}.

\begin{prop}\label{dalpha}
Assume Hypothesis \ref{STR1}\eqref{STR1.1} holds true. If $f:\X\ra\R$ is a Fr\'echet differentiable function with continuous derivative operator, then $f$ is also $H_R$-differentiable with continuous $H_R$-derivative operator and, for every $x\in\X$, it holds $\nabla_Rf(x)=R^2\nabla f(x)$. 
\end{prop}

Let $Y$ be a Banach space endowed with the norm $\norm{\cdot}_Y$ and $\alpha\in(0,1)$. We define the space of $Y$-valued, bounded and uniformly continuous H\"older functions along the directions of $H_R$ as
\begin{align}\label{defn_Calpha}
{\rm BUC}_R^\alpha(\X;Y):=\set{f\in{\rm BUC}(\X)\tc [f]_{{\rm BUC}_R^\alpha(\X;Y)}:=\sup_{\substack{x\in\X\\ h\in H_R\setminus\set{0}}}\frac{\|f(x+h)-f(x)\|_Y}{\|h\|^\alpha_{R}}<+\infty},
\end{align}
endowed with the norm
\begin{align*}
\|f\|_{{\rm BUC}_R^\alpha(\X;Y)}:=\norm{f}_\infty+[f]_{{\rm BUC}_R^\alpha(\X;Y)}.
\end{align*}
As usual, if $Y=\R$ we simply write ${\rm BUC}_R^\alpha(\X)$.

For $k\in\N$, we denote by ${\rm BUC}^{k}_{R}(\X)$ the space of bounded, uniformly continuous and $k$-times $H_R$-differentiable functions $f:\X\ra\R$ such that $\D^i_Rf\in {\rm BUC}(\X;\mathcal{L}^{(i)}(H_R;\R))$. We endow ${\rm BUC}^{k}_{R}(\X)$  with the norm
\[
\norm{f}_{{\rm BUC}^{k}_{R}(\X)}:=\norm{f}_{\infty}+\sum^k_{i=1}\sup_{x\in\X}\|\J^i_Rf(x)\|_{\mathcal{L}^{(i)}(H_R;\R)}.
\]
For $k\in\N$ and $\alpha\in (0,1)$, we denote by ${\rm BUC}^{k+\alpha}_{R}(\X)$ the subspace of ${\rm BUC}^{k}_{R}(\X)$ of functions $f:\X\ra\R$ such that $\D^k_Rf\in {\rm BUC}^{\alpha}_{R}(\X;\mathcal{L}^{(k)}(H_R;\R))$. We endow ${\rm BUC}^{k+\alpha}_{R}(\X)$ with the norm
\[
\norm{f}_{{\rm BUC}^{k+\alpha}_{R}(\X)}:= \norm{f}_{{\rm BUC}^{k}_{R}(\X)}+[D_R^kf]_{{\rm BUC}^{\alpha}_{R}(\X;\mathcal{L}^{(k)}(H_R;\R))},
\]
where the seminorm $[\cdot]_{{\rm BUC}^{\alpha}_{R}(\X;\mathcal{L}^{(k)}(H_R;\R))}$ is defined in \eqref{defn_Calpha}. If $R={\rm Id}_{\X}$, then we will simply omit the subscript $R$ from the notations.

Now we can state one of the main result of this paper. 

\begin{thm}\label{THM_SCHAUDER}
Assume Hypotheses \ref{STR2} hold true. For any $\lambda>0$, $\alpha\in(0,1)$ and $f\in{\rm BUC}^\alpha_R(\X)$ the solution $u$ of \eqref{intr_problema_stazionario}, introduced in \eqref{GS}, belongs to ${\rm BUC}^{2+\alpha}_R(\X)$, and there exists a positive constant $C$ depending only on $\lambda$ and $\alpha$ such that
\begin{align}\label{THM_SCHAUDER_STIMA}
\|u\|_{{\rm BUC}^{2+\alpha}_R(\X)}\leq C\norm{f}_{{\rm BUC}^\alpha_R(\X)}.
\end{align}
\end{thm}


Theorem \ref{THM_SCHAUDER} says nothing about the case $f\in {\rm BUC}(\X)$, i.e. $\alpha=0$. To study this case we will introduce an appropriate Zygmund space. Indeed the second main result we will show (Theorem \ref{Thm_Zyg}) states that if Hypotheses \ref{STR2} hold true and $f$ belongs to ${\rm BUC}(\X)$, then the solution $u$ of \eqref{intr_problema_stazionario} is bounded, uniformly continuous and $H_R$-differentiable with bounded and uniformly continuous $H_R$-gradient $\J_Ru$ such that
\begin{align*}
\sup_{\substack{x\in\X\\ h\in H_R\setminus\set{0}}}\frac{\|\D_Ru(x+2h)-2\D_Ru(x+h)+\D_Ru(x)\|_{R}}{\|h\|_R}<+\infty,
\end{align*}
namely $\J_Ru$ satisfies a Zygmund type condition along $H_R$. 

Moreover we show Schauder type regularity results (Theorem \ref{Thm_evol}) for the mild solution $v$, introduced in \eqref{mild_evol}, of the evolution equation \eqref{evol_prob}. Finally in Proposition \ref{SCHvar} we will prove a Schauder regularity results for the solution $u$ of the stationary equation
\begin{equation*}
\lambda u-L u-\langle F,\J_Ru\rangle_R=f,\qquad \lambda>0,\ f\in {\rm BUC}_R^{\alpha}(\X),
\end{equation*}
where $F$ is a function belonginig to ${\rm BUC}^\alpha_R(\X;H_R)$ with suitable small norm, and  $L$ is the weak generator in ${\rm BUC}(\X)$ of the transition semigroup associated to 
\begin{gather*}
\eqsys{
dX(t,x)=AX(t,x)dt+RdW(t), & t>0;\\
X(0,x)=x\in \X,
}
\end{gather*} 
where $A$ and $R$ satisfy Hypotheses \ref{STR2}.

\begin{rmk}
In \cite{CL19} the authors study the case with $A=-(1/2){\rm Id}_{\X}$, $G\equiv 0$ and $R=Q^{1/2}$, where $Q$ is a non-negative, self-adjoint and trace class operator, hence $R=Q^{1/2}$ is not necessary injective. Without the injectivity of $R$ we are not aware of a result similar to Proposition \ref{dalpha} and this may give problems in the proof of Theorem \ref{BEL}, particularly in the step from \eqref{raaaa} to \eqref{raaaa2}. Instead in  \cite{CL19} they can exploit the Mehler representation formula of the semigroup to avoid this problem. In \cite{CL19}, again exploiting the Mehler representation formula of the semigroup, they obtain estimates similar to the ones in Proposition \ref{Holder}, without the use of Theorem \ref{intDS}. Doing so they are able to prove (see \cite[Theorem 3.4 and Theorem 3.6]{CL19}) results similar to Theorem \ref{THM_SCHAUDER} replacing the ${\rm BUC}(\X)$ space with $C_b(\X)$. Unfortunately, we are not able to prove the estimates in Proposition \ref{Holder} without the results of Section \ref{Interpolation}. Furthermore, we do not know whether Theorem \ref{intDS} holds true if we replace ${\rm BUC}(\X)$ by $C_b(\X)$.
\end{rmk}

\section{An interpolation result}\label{Interpolation}
In the following we will prove an interpolation result which will be fundamental in the next sections. A similar result appears in \cite{CDP96-2}, but we stress that the derivative operator considered in \cite{CDP96-2} is different from the one introduced in this paper (Definition \ref{intr_defn_diff}). Let $\mathfrak{X}$ be the set of functions $\varphi$ belonging to ${\rm BUC}(\X)$ such that they are $H_R$-differentiable and 
\begin{align*}
\|\D_R\varphi\|_\infty:=\sup_{x\in\X}\|\D_R\varphi(x)\|_{\mathcal{L}(H_R;\R)}<+\infty.
\end{align*}
We endowed $\mathfrak{X}$ with the norm
\(
\|\varphi\|_{\mathfrak{X}}:=\|\varphi\|_\infty+\|\D_R\varphi\|_\infty.
\)

\begin{prop}
Assume Hypothesis \ref{STR1}\eqref{STR1.1} holds true. $(\mathfrak{X},\norm{\cdot}_{\mathfrak{X}})$ is a Banach space.
\end{prop}

\begin{proof}
The proof is standard, we just give a brief sketch. The fact that $\mathfrak{X}$ is a normed space is obvious. Now let $(\varphi_n)_{n\in\N}\subseteq \mathfrak{X}$ be a Cauchy sequence in $\mathfrak{X}$. Let $\varphi\in{\rm BUC}(\X)$ and $L:\X\ra \mathcal{L}(H_R;\R)$ be the uniform limit of $(\varphi_n)_{n\in\N}$ and of $(\D_R\varphi_n)_{n\in\N}$. By \cite[Theorem 7.11]{Rud76} we get that 
\begin{align*}
\lim_{\|h\|_R\ra 0}\frac{|\varphi(x+h)-\varphi(x)-L(x)h|}{\|h\|_R} &=\lim_{\|h\|_R\ra 0}\lim_{n\ra+\infty}\frac{|\varphi_n(x+h)-\varphi_n(x)-\D_R\varphi_n(x)h|}{\|h\|_R}\\
&=\lim_{n\ra+\infty}\lim_{\|h\|_R\ra 0}\frac{|\varphi_n(x+h)-\varphi_n(x)-\D_R\varphi_n(x)h|}{\|h\|_R}=0.
\end{align*}
This conclude the proof.
\end{proof}

We introduce a modification of the Lasry--Lions approximating procedure introduced in \cite{LL86}.

\begin{prop}\label{Lex}
Assume Hypothesis \ref{STR1}\eqref{STR1.1} holds true and let $f\in {\rm BUC}^\alpha_R(\X)$ for some $\alpha\in(0,1)$. For every $\eps>0$, we consider the function
\begin{align}\label{LL_definition}
f_\eps(x):=\sup_{h\in H_R}\set{\inf_{k\in H_R}\set{f(x+h-k)+\frac{1}{2\eps}\|k\|_R^2}-\frac{1}{\eps}\|h\|_R^2},\qquad x\in\X.
\end{align}
For every $\eps>0$, the function $f_\eps$ belongs to $\mathfrak{X}$ and for any $x\in\X$
\begin{align}
\|f_\eps\|_\infty&\leq\|f\|_\infty.\label{LL_limitatezza}
\end{align}
Furthermore there exists a positive constant $c_\alpha$, depending only on $\alpha$, such that
\begin{align}
0\leq f(x)-f_\eps(x)&\leq c_\alpha [f]_{{\rm BUC}_R^\alpha(\X)}^{2/(2-\alpha)}\eps^{\alpha/(2-\alpha)},\qquad x\in\X,\label{LL_approximation}
\end{align}
and
\begin{align}
\|\J_Rf_\eps\|_\infty &\leq 2\sqrt{2}c^{1/2}_\alpha [f]_{{\rm BUC}_R^\alpha(\X)}^{1/(2-\alpha)}\eps^{(\alpha-1)/(2-\alpha)}.\label{LL_derivative_estimate}
\end{align}
\end{prop}

\begin{proof}
Throughout the proof we fix $\eps>0$. We start by showing that \eqref{LL_limitatezza} holds true. Indeed, for any $x\in\X$, ot holds
\begin{align}
f_\eps(x)&=\sup_{h\in H_R}\set{\inf_{k\in H_R}\set{f(x+h-k)+\frac{1}{2\eps}\|k\|_R^2}-\frac{1}{\eps}\|h\|_R^2}\notag\\
&\leq \sup_{h\in H_R}\set{f(x)+\frac{1}{2\eps}\|h\|_R^2-\frac{1}{\eps}\|h\|_R^2}\leq f(x)\leq \|f\|_\infty.\label{Gerald1}
\end{align}
In a similar way, for every $x\in\X$, we get
\begin{align}
f_\eps(x)&=\sup_{h\in H_R}\set{\inf_{k\in H_R}\set{f(x+h-k)+\frac{1}{2\eps}\|k\|_R^2}-\frac{1}{\eps}\|h\|_R^2}\notag\\
&\geq \inf_{k\in H_R}\set{f(x-k)+\frac{1}{2\eps}\|k\|_R^2}\geq -\norm{f}_\infty.\label{Gerald2}
\end{align}
By \eqref{Gerald1} and \eqref{Gerald2} we get \eqref{LL_limitatezza}.

Now we prove that $f_\eps$ is uniformly continuous. Since $f$ is uniformly continuous we know that for every $\eta>0$ there exists $\delta:=\delta(\eta)>0$ such that for every $x,y\in\X$ with $|x-y|<\delta$ it holds $|f(x)-f(y)|<\eta$. Let $x,y\in\X$ satisfying $|x-y|<\delta$, then for every $\sigma>0$ there exists $h_\sigma,k_\sigma\in H_R$ such that
\begin{align*}
f_\eps(x)-f_\eps(y)&\leq \inf_{k\in H_R}\set{f(x+h_\sigma-k)+\frac{1}{2\eps}\|k\|_R^2}-\frac{1}{\eps}\|h_\sigma\|_R^2+\sigma\\
&\qquad-\inf_{k\in H_R}\set{f(y+h_\sigma-k)+\frac{1}{2\eps}\|k\|_R^2}+\frac{1}{\eps}\|h_\sigma\|_R^2\\
&\leq f(x+h_\sigma-k_\sigma)+\frac{1}{2\eps}\|k_\sigma\|_R^2-f(y+h_\sigma-k_\sigma)-\frac{1}{2\eps}\|k_\sigma\|_R^2+2\sigma\\
&\leq \eta+2\sigma.
\end{align*}
In a similar way we get that $f_\eps(x)-f_\eps(y)\geq -\eta-2\sigma$. So $f_\eps$ is uniformly continuous.

It is now time to prove \eqref{LL_approximation}. Let $x\in\X$, for every $\eta>0$ there exists $k_\eta\in H_R$ such that
\begin{align}
0\leq f(x)-f_\eps(x)&\leq f(x)-\inf_{k\in H_R}\set{f(x-k)+\frac{1}{2\eps}\|k\|_R^2}\notag\\
&\leq f(x)-f(x-k_\eta)-\frac{1}{2\eps}\|k_\eta\|_R^2+\eta\notag\\
&\leq [f]_{{\rm BUC}_R^\alpha(\X)}\|k_\eta\|_R^\alpha-\frac{1}{2\eps}\|k_\eta\|_R^2+\eta.\label{Regan}
\end{align}
Before proceeding we need to estimate $\|k_\eta\|_R$. By \eqref{Regan} we get
\begin{align*}
\|k_\eta\|_R^2\leq 2\eps [f]_{{\rm BUC}_R^\alpha(\X)}\|k_\eta\|_R^\alpha+2\eps \eta.
\end{align*}
By the Young inequality, for every $c>0$
\begin{align*}
\|k_\eta\|_R^2\leq \frac{\alpha}{2}c^{2/\alpha} \|k_\eta\|_R^2+\frac{2-\alpha}{2}\frac{1}{c^{2/(2-\alpha)}}  (2\eps [f]_{{\rm BUC}_R^\alpha(\X)})^{2/(2-\alpha)}+2\eps \eta.
\end{align*}
Now taking $c=\alpha^{-\alpha/2}$ we get 
\begin{align}\label{Truman}
\|k_\eta\|_R^2\leq (2-\alpha)\alpha^{\alpha/(2-\alpha)}2^{2/(2-\alpha)}[f]_{{\rm BUC}_R^\alpha(\X)}^{2/(2-\alpha)}\eps^{2/(2-\alpha)}+4\eps\eta.
\end{align}
Now using \eqref{Regan} and \eqref{Truman} we get, for any $x\in\X$
\begin{align*}
0\leq f(x)-f_\eps(x)&\leq [f]_{{\rm BUC}_R^\alpha(\X)}\big((2-\alpha)\alpha^{\alpha/(2-\alpha)}2^{2/(2-\alpha)}[f]_{{\rm BUC}_R^\alpha(\X)}^{2/(2-\alpha)}\eps^{2/(2-\alpha)}+4\eps\eta\big)^{\alpha/2}+\eta.
\end{align*}
Since the above estimate holds for every $\eta>0$ we get \eqref{LL_approximation}.

The proof of the existence of the $H_R$-derivative of $f_\eps$ is based upon a known result that state that if a function $u$ is both semiconvex and semiconcave, then $u$ is Fr\'echet differentiable with Lipschitz continuous Fr\'echet derivate operator (we refer to \cite[Theorem 3.1]{KZ21} for a proof of this result). The proof follows the same lines of the one in \cite[p. 264]{LL86} replacing the function $\underline{u}_{\eps,\delta}$ with the function $h\mapsto f_\eps(x+h)$. This is due to the fact that, for every $x\in\X$, if we consider $u_x:H_R\ra\R$ defined as 
\[u_x(h):=f(x+h),\qquad h\in H_R,\] 
then, for every $x\in \X$, the function $k\mapsto f_\eps(x+k)$ coincides with the classical Lasry--Lions approximation of the function $u_x$, defined in \cite{LL86}, evaluated at $k$. This gives us that $\D_R f_\eps(x)$ exists for every $x\in\X$ and it satisfies
\begin{equation}\label{lipgradiente}
\|\D_Rf_\eps(x+h)-\D_Rf_\eps(x+k)\|_{\mathcal{L}(H_R;\R)}\leq L\|h-k\|_R,\qquad h,k\in H_R,
\end{equation}
for some positive constant $L$, independent of $h,k\in H_R$.

We are now in the position to prove \eqref{LL_derivative_estimate}. Let $x\in\X$ and observe that for every $\sigma>0$ there exists $h_\sigma\in H_R$ such that
\begin{align*}
f_\eps(x)\leq \inf_{k\in H_R}\set{f(x+h_\sigma-k)+\frac{1}{2\eps}\|k\|_R^2}-\frac{1}{\eps}\|h_\sigma\|_R^2+\sigma.
\end{align*}
A straightforward calculation gives us
\begin{align*}
\frac{1}{\eps}\|h_\sigma\|_R^2\leq f(x)-f_\eps(x)+\sigma+\frac{1}{2\eps}\|h_\sigma\|_R^2,\qquad x\in\X.
\end{align*}
So by \eqref{LL_approximation} we obtain
\begin{align}\label{Discipline}
\|h_\sigma\|_R^2\leq 2 c_\alpha[f]_{{\rm BUC}_R^\alpha(\X)}^{2/(2-\alpha)}\eps^{2/(2-\alpha)}+2\eps\sigma.
\end{align}
By \eqref{Discipline} we get for $x\in\X$ and $h\in H_R$
\begin{align}
f_\eps(x+h)-f_\eps(x)&\leq \inf_{k\in H_R}\set{f(x+h+h_\sigma-k)+\frac{1}{2\eps}\|k\|_R^2}-\frac{1}{\eps}\|h_\sigma\|_R^2+\sigma\notag\\
&\qquad -\inf_{k\in H_R}\set{f(x+h+h_\sigma-k)+\frac{1}{2\eps}\|k\|_R^2}+\frac{1}{\eps}\|h+h_\sigma\|_R^2\notag\\
&=\frac{1}{\eps}\|h+h_\sigma\|_R^2-\frac{1}{\eps}\|h_\sigma\|_R^2+\sigma=\frac{1}{\eps}\|h\|_R^2+\frac{2}{\eps}\langle h,h_\sigma\rangle_R+\sigma\notag\\
&\leq \frac{1}{\eps}\|h\|_R^2+\frac{2}{\eps}\|h\|_R(2 c_\alpha[f]_{{\rm BUC}_R^\alpha(\X)}^{2/(2-\alpha)}\eps^{2/(2-\alpha)}+2\eps\sigma)^{1/2}+\sigma.\label{supplex}
\end{align}
Since \eqref{supplex} holds for every $\sigma>0$, taking the infimum we get for any $x\in\X$ and $h\in H_R$
\begin{align*}
f_\eps(x+h)-f_\eps(x)\leq \frac{1}{\eps}\|h\|_R^2+2\|h\|_R(2 c_\alpha[f]_{{\rm BUC}_R^\alpha(\X)}^{2/(2-\alpha)})^{1/2}\eps^{(\alpha-1)/(2-\alpha)}.
\end{align*}
A standard argument concludes the proof.
\end{proof}

We shall use the $K$ method for real interpolation spaces (see, for example, \cite{Lun18,Tri95}). Let $\K_1$ and $\K_2$ be two Banach spaces, with norms $\norm{\cdot}_{\K_1}$ and $\norm{\cdot}_{\K_2}$, respectively. If $\K_2\subseteq \K_1$ with a continuous embedding and both of them are continuously embedded in a Hausdorff topological vector space $\mathcal{V}$, then for every $r>0$ and $x\in \K_1+\K_2$ we define
\begin{align*}
K(r,x):=\inf\set{\|a\|_{\K_1}+r\|b\|_{\K_2}\tc x=a+b,\ a\in \K_1,\ b\in \K_2}.
\end{align*}
For any $\vartheta\in(0,1)$, we set
\begin{align}
(\K_1,\K_2)_{\vartheta,\infty}&:=\left\{x\in \K_1+\K_2\,\middle|\, \|x\|_{(\K_1,\K_2)_{\vartheta,\infty}}:=\sup_{r>0}r^{-\vartheta}K(r,x)<+\infty\right\}.\label{defn_norm_interp}
\end{align}
It is standard to show that $(\K_1,\K_2)_{\vartheta,\infty}$ endowed with the norm $\norm{\cdot}_{(\K_1,\K_2)_{\vartheta,\infty}}$ is a Banach space.

\begin{thm}\label{intDS}
Assume Hypothesis \ref{STR1}\eqref{STR1.1} holds true and let $\alpha\in(0,1)$. 
Up to an equivalent renorming, it holds $({\rm BUC}(\X),\mathfrak{X})_{\alpha,\infty}={\rm BUC}_R^\alpha(\X)$.
\end{thm}

\begin{proof}
We start by showing that $({\rm BUC}(\X),\mathfrak{X})_{\alpha,\infty}\subseteq {\rm BUC}_R^\alpha(\X)$. For every $\varphi\in ({\rm BUC}(\X),\mathfrak{X})_{\alpha,\infty}$ and any $r,\eps>0$ there exist $f_{r,\eps}\in{\rm BUC}(\X)$ and $g_{r,\eps}\in\mathfrak{X}$ such that 
\begin{align*}
\varphi(x)=f_{r,\eps}(x)+g_{r,\eps}(x),\qquad x\in\X;
\end{align*}
and
\begin{align}\label{ceffo}
\|f_{r,\eps}\|_\infty+r\|g_{r,\eps}\|_{\mathfrak{X}}\leq r^\alpha\|\varphi\|_{({\rm BUC}(\X),\mathfrak{X})_{\alpha,\infty}}+\eps.
\end{align}
So by the mean value theorem and \eqref{ceffo} we get for $x\in\X$ and $h\in H_R$
\begin{align}
|\varphi(x+h)-\varphi(x)|&\leq 2\|f_{r,\eps}\|_\infty+|g_{r,\eps}(x+h)-g_{r,\eps}(x)|\notag\\
&\leq 2\|f_{r,\eps}\|_\infty+\|\D_R g_{r,\eps}\|_\infty\|h\|_R\notag\\
&\leq  2 r^\alpha \|\varphi\|_{({\rm BUC}(\X),\mathfrak{X})_{\alpha,\infty}}+2\eps+r^{\alpha-1}\|\varphi\|_{({\rm BUC}(\X),\mathfrak{X})_{\alpha,\infty}}\|h\|_R+\eps r^{-1}\|h\|_R.\label{deodorant}
\end{align}
Now letting $\eps$ tend to zero in \eqref{deodorant} and setting $r=\|h\|_R$ we get
\begin{align*}
|\varphi(x+h)-\varphi(x)|&\leq 3\|\varphi\|_{({\rm BUC}(\X),\mathfrak{X})_{\alpha,\infty}}\|h\|_R^\alpha,\qquad x\in\X,\ h\in H_R.
\end{align*}
This gives us the continuous embedding $({\rm BUC}(\X),\mathfrak{X})_{\alpha,\infty}\subseteq {\rm BUC}_R^\alpha(\X)$.

Now we prove that ${\rm BUC}_R^\alpha(\X)\subseteq ({\rm BUC}(\X),\mathfrak{X})_{\alpha,\infty}$. Let $\varphi\in{\rm BUC}_R^\alpha(\X)$ and for every $\eps>0$ let $\varphi_\eps$ be the function defined in \eqref{LL_definition}. For $r\in (0,1)$ we consider the functions $f_r:\X\ra\R$ and $g_r:\X\ra\R$ defined as
\begin{align*}
f_r(x):=\varphi(x)-\varphi_{r^{2-\alpha}}(x),\qquad g_r(x):=\varphi_{r^{2-\alpha}}(x),\qquad x\in\X.
\end{align*}
By \eqref{LL_approximation} we get that there exists a constant $k_1=k_1(\alpha,\varphi)>0$ such that 
\begin{align*}
\|f_r\|_\infty\leq k_1 r^\alpha.
\end{align*}
By \eqref{LL_limitatezza} and \eqref{LL_derivative_estimate}, there exist a constant $k_2=k_2(\alpha,\varphi)>0$ such that
\begin{align*}
\|g_r\|_{\mathfrak{X}}&=\|\varphi_{r^{2-\alpha}}\|_\infty+\|\D_R\varphi_{r^{2-\alpha}}\|_\infty\leq k_2r^{\alpha-1}
\end{align*}
So for every $r\in(0,1)$
\begin{align*}
K(r,\varphi)\leq (k_1+k_2)r^{\alpha}.
\end{align*}
Observe that the above estimate is trivial in the case that $r>1$. Recalling \eqref{defn_norm_interp} we get the thesis.
\end{proof}

\begin{rmk}\label{Cell}
We stress that \eqref{lipgradiente} does not guarantee the uniformly continuity of the $H_R$-derivative operator with respect to $x$. On the other hand if $R={\rm Id}_\X$ then, by \eqref{lipgradiente}, the Lasry--Lions type approximants $f_\epsilon$ belongs to ${\rm BUC}^1(\X)$. Hence Proposition \ref{Lex} and Theorem \ref{intDS} hold true with $\mathfrak{X}$ replaced with ${\rm BUC}^1(\X)$. This result was already present in \cite{CDP96-2}.
\end{rmk}

We conclude this section by recalling a classical interpolation result that we will use in the paper (see \cite[Theorem 1.12 of Chapter 5]{CR88} for a proof).


\begin{thm}\label{classic}
Let $\K_0,\K_1,\mathcal{H}_0,\mathcal{H}_1$ be Banach spaces such that $\mathcal{H}_0\subseteq \mathcal{K}_0$ and $\mathcal{H}_1\subseteq \mathcal{K}_1$ with continuous embeddings. If $T$ is a linear mapping such that $T:\K_0\ra\K_1$ and $T:\mathcal{H}_0\ra\mathcal{H}_1$ and for some $N_{\K},N_{\mathcal{H}}>0$ it hold
\begin{align*}
\|Tx\|_{\K_1}&\leq N_{\K}\|x\|_{\K_0},\qquad x\in \K_0;\\
\|Ty\|_{\mathcal{H}_1}&\leq N_{\mathcal{H}}\|y\|_{\mathcal{H}_0},\qquad y\in \mathcal{H}_0,
\end{align*}
then, for every $\vartheta\in(0,1)$, $T$ maps $(\K_0,\mathcal{H}_0)_{\vartheta,\infty}$ in $(\K_1,\mathcal{H}_1)_{\vartheta,\infty}$ and
\begin{align*}
\|Tx\|_{(\K_1,\mathcal{H}_1)_{\vartheta,\infty}}&\leq N_{\K}^{1-\vartheta}N_{\mathcal{H}}^{\vartheta}\|x\|_{(\K_0,\mathcal{H}_0)_{\vartheta,\infty}},\qquad x\in (\K_0,\mathcal{H}_0)_{\vartheta,\infty}.
\end{align*}
\end{thm}

\section{$H_R$ regularity of the mild solution $\{X(t,x)\}_{t\geq 0}$}\label{Regularity_Mild}

In this section we are interested in investigate the differentiability properties of the mild solution of \eqref{eqF02}. For every $x\in\X$, by Hypotheses \ref{STR1} and \cite[Theorem 7.5]{DA-ZA4}, \eqref{eqF02} has a unique mild solution $\{X(t,x)\}_{t\geq 0}$. 
Recall that the map $x\mapsto X(\cdot,x)$ is Lipschitz continuous uniformly with respect to the time variable (see \cite[Proposition 3.13]{Big21} or \cite[Proposition 3.7]{Mas08} for a proof). More precisely, for every $T>0$, there exists a positive constant $\eta=\eta(T)$ such that
\begin{align}\label{Lip_mild}
\sup_{t\in[0,T]}\|X(t,x)-X(t,y)\|\leq\eta \|x-y\|,\qquad x,y\in\X.
\end{align}
By \cite[Section 2]{DA5} (still using Hypotheses \ref{STR1}) for any $t>0$ and for $\mathbb{P}$-a.e. $\omega\in\Omega$ the map $x\mapsto X(t,x)(w)$ is three times Gateaux differentiable and if we set for $x,h,k,j\in\X$ and $t>0$
\begin{align*}
\delta^h_1(t,x)&:=\J_GX(t,x)h;\\
\delta^{h,k}_2(t,x)&:=\J^2_GX(t,x)(h,k);\\
\delta^{h,k,j}_3(t,x)&:=\J^3_GX(t,x)(h,k,j),
\end{align*}
then the processes $\{\delta^h_1(t,x)\}_{t\geq 0}$, $\{\delta^{h,k}_2(t,x)\}_{t\geq 0}$ and $\{\delta^{h,k,j}_3(t,x)\}_{t\geq 0}$ are the mild solutions of  
\begin{align}\label{eqdelta1}
&\eqsys{
d\delta^h_1(t,x)=[A+R\J G(X(t,x))]\delta^h_1(t,x)dt, & t>0;\\
\delta^h_1(0,x)=h;
}\\
\label{eqdelta2}
&\eqsys{
d\delta^{h,k}_2(t,x)=\big([A+R\J G(X(t,x))]\delta^{h,k}_2(t,x)\\
\qquad\qquad\qquad + R\J^2 G(X(t,x))(\delta^h_1(t,x),\delta^k_1(t,x))\big) dt, & t>0;\\
\delta^{h,k}_2(0,x)=0;
}
\\
\label{eqdelta3}
&\eqsys{
d\delta^{h,k,j}_3(t,x)=\big([A+R\J G(X(t,x))]\delta^{h,k,j}_3(t,x)\\ 
\qquad\qquad\qquad\, + R\J^2 G(X(t,x))(\delta^j_1(t,x),\delta^{h,k}_2(t,x))\\ 
\qquad\qquad\qquad\, + R\J^2 G(X(t,x))(\delta^{h,j}_2(t,x),\delta^{k}_1(t,x))\\ 
\qquad\qquad\qquad\, + R\J^2 G(X(t,x))(\delta^h_1(t,x),\delta^{j,k}_2(t,x))\\ 
\qquad\qquad\qquad\, + R\J^3 G(X(t,x))(\delta^{h}_1(t,x),\delta^{k}_1(t,x),\delta^{j}_1(t,x))\big)dt, & t>0;\\ 
\delta^{h,k,j}_3(0,x)=0.
}
\end{align} 

The first proposition we will prove concerns some uniform estimates, with respect to $x\in\X$, of the processes $\{\delta^h_1(t,x)\}_{t\geq 0}$, $\{\delta^{h,k}_2(t,x)\}_{t\geq 0}$ and $\{\delta^{h,k,j}_3(t,x)\}_{t\geq 0}$.

\begin{prop}\label{Hreg}
Assume that Hypotheses \ref{STR2} hold true. For every $x\in\X$ and $h,k,j\in H_R$ the processes $\{\delta^h_1(t,x)\}_{t\geq 0}$, $\{\delta^{h,k}_2(t,x)\}_{t\geq 0}$ and $\{\delta^{h,k,j}_3(t,x)\}_{t\geq 0}$ are $H_R$-valued. Moreover there exist two positive constants $M_2$ and $M_3$ such that for every $x\in\X$, $h,k,j\in H_R$ and $t>0$, it holds $\mathbb{P}$-a.e.
\begin{align}
\|\delta^h_1(t,x)\|_R&\leq e^{\zeta_R t}\norm{h}_R;\label{1R}\\
\|\delta^{h,k}_2(t,x)\|_R&\leq M_2 K_1(t,\zeta_R)\norm{h}_R\norm{k}_R;\label{2R}\\
\|\delta^{h,k,j}_3(t,x)\|_R&\leq M_3 K_2(t,\zeta_R)\norm{h}_R\norm{k}_R\norm{j}_R,\label{3R}
\end{align}
where
\begin{align}
K_1(t,\zeta_R)&:=
\eqsys{
\zeta_R^{-1}e^{t\zeta_R}(e^{t\zeta_R}-1),& \zeta_R\neq 0;\\
t,& \zeta_R=0;
}\label{defn_K1}\\
K_2(t,\zeta_R)&:=
\eqsys{
(2\zeta^2_R)^{-1}e^{t\zeta_R}(e^{t\zeta_R}-1)((1+\zeta_R)e^{t\zeta_R}+\zeta_R-1),& \zeta_R\neq 0;\\
2^{-1}t^2+t,& \zeta_R=0.
}\label{defn_K2}
\end{align}
\end{prop}

\begin{proof}
All estimates in this proof are meant almost everywhere with respect to $\mathbb{P}$. The fact that the processes $\{\delta^h_1(t,x)\}_{t\geq 0}$, $\{\delta^{h,k}_2(t,x)\}_{t\geq 0}$ and $\{\delta^{h,k,j}_3(t,x)\}_{t\geq 0}$ are $H_R$-valued are easy consequences of the mild form of \eqref{eqdelta1}, \eqref{eqdelta2} and \eqref{eqdelta3}. So we just need to prove the estimates \eqref{1R}, \eqref{2R} and \eqref{3R}.

If $\{\delta^h_1(t,x)\}_{t\geq 0}$, $\{\delta^{h,k}_2(t,x)\}_{t\geq 0}$ and $\{\delta^{h,k,j}_3(t,x)\}_{t\geq 0}$ are equal to zero, then \eqref{1R}, \eqref{2R} and \eqref{3R} are obvious, so we can fix $t>0$, $x\in\X$ and $h,k,j\in H_R$ such that the processes $\{\delta^h_1(t,x)\}_{t\geq 0}$, $\{\delta^{h,k}_2(t,x)\}_{t\geq 0}$ and $\{\delta^{h,k,j}_3(t,x)\}_{t\geq 0}$ are non-zero. We assume that the processes $\{\delta^h_1(t,x)\}_{t\geq 0}$, $\{\delta^{h,k}_2(t,x)\}_{t\geq 0}$ and $\{\delta^{h,k,j}_3(t,x)\}_{t\geq 0}$ are strict solutions of \eqref{eqdelta1}, \eqref{eqdelta2} and \eqref{eqdelta3} respectively, otherwise we proceed as in \cite{Big21} or \cite[Proposition 6.2.2]{CER1} approximating $\{\delta^h_1(t,x)\}_{t\geq 0}$, $\{\delta^{h,k}_2(t,x)\}_{t\geq 0}$ and $\{\delta^{h,k,j}_3(t,x)\}_{t\geq 0}$ by means of sequences of more regular processes.

We start by proving \eqref{1R}. Fix $t>0$ and $x\in\X$, scalarly multiplying the uppermost equation in \eqref{eqdelta1} by $\delta^h_1(t,x)$ and using \eqref{dissR} we obtain
\begin{align*}
\frac{1}{2}\frac{d}{dt}\|\delta^h_1(t,x)\|^2_R=\langle [A+R\J G(X(t,x))]\delta^h_1(t,x),\delta^h_1(t,x)\rangle_R\leq \zeta_R\|\delta_1^h(t,x)\|^2_R.
\end{align*}
By a standart argument we obtain \eqref{1R}. 

Now we take care of the proof of \eqref{2R}. Fix $t>0$ and $x\in\X$, scalarly multiplying the uppermost equation in \eqref{eqdelta2} by $\delta^{h,k}_2(t,x)$ and using Hypothesis \ref{STR1}\eqref{STR1.3}, \eqref{dissR} and \eqref{1R} we get
\begin{align}
\scal{\frac{d}{dt}\delta^{h,k}_2(t,x)}{\delta^{h,k}_2(t,x)}_R&=\langle [A+R\J G(X(t,x))]\delta^{h,k}_2(t,x),\delta^{h,k}_2(t,x)\rangle_R\notag\\
&\qquad\qquad+\langle R\J^2G(X(t,x))(\delta^{h}_1(t,x),\delta^{k}_1(t,x)),\delta^{h,k}_2(t,x)\rangle_R\notag\\
&\leq \zeta_R\|\delta^{h,k}_2(t,x)\|_R^2+M\|R\|_{\mathcal{L}(\X)}^2 e^{2t\zeta_R}\norm{h}_R\norm{k}_R\|\delta_2^{h,k}(t,x)\|_R,\label{op}
\end{align}
where $M$ is the constant appearing in Hypothesis \ref{STR1}\eqref{STR1.3}. Recalling that, for every $x\in\X$ and $h,k\in H_R$, by \cite[Proposition A.1.3]{CER1}, the map $t\mapsto \|\delta^{h,k}_2(t,x)\|_R$ is differentiable, we can divide \eqref{op} by $\|\delta^{h,k}_2(t,x)\|_R$ to get
\begin{align*}
\frac{d}{dt}\|\delta^{h,k}_2(t,x)\|_R\leq\zeta_R\|\delta^{h,k}_2(t,x)\|_R+M\|R\|_{\mathcal{L}(\X)}^2 e^{2t\zeta_R}.
\end{align*}
By a standard comparison result we obtain \eqref{2R} with $M_2:= M\|R\|_{\mathcal{L}(\X)}^2$.

It is now time to prove \eqref{3R}. Fix $t>0$ and $x\in\X$, scalarly multiplying the uppermost equation of \eqref{eqdelta3} by $\delta^{h,k,j}_3(t,x)$, and then using Hypothesis \ref{STR1}\eqref{STR1.3}, \eqref{dissR}, \eqref{1R} and \eqref{2R} we obtain
\begin{align*}
&\scal{\frac{d}{dt}\delta^{h,k,j}_3(t,x)}{\delta^{h,k,j}_3(t,x)}_R=\langle[A+R\J G(X(t,x))]\delta^{h,k,j}_3(t,x),\delta^{h,k,j}_3(t,x)\rangle_R\\
&\qquad\qquad\qquad\qquad\qquad\qquad\quad\ \ \ + \langle R\J^2 G(X(t,x))(\delta^j_1(t,x),\delta^{h,k}_2(t,x)),\delta^{h,k,j}_3(t,x)\rangle_R\\ 
&\qquad\qquad\qquad\qquad\qquad\qquad\quad\ \ \ + \langle R\J^2 G(X(t,x))(\delta^{h,j}_2(t,x),\delta^{k}_1(t,x)),\delta^{h,k,j}_3(t,x)\rangle_R\\ 
&\qquad\qquad\qquad\qquad\qquad\qquad\quad\ \ \ + \langle R\J^2 G(X(t,x))(\delta^h_1(t,x),\delta^{j,k}_2(t,x)),\delta^{h,k,j}_3(t,x)\rangle_R\\ 
&\qquad\qquad\qquad\qquad\qquad\qquad\quad\ \ \ + \langle R\J^3 G(X(t,x))(\delta^{h}_1(t,x),\delta^{k}_1(t,x),\delta^{j}_1(t,x)),\delta^{h,k,j}_3(t,x)\rangle_R\\
&\qquad \leq\zeta_R\|\delta^{h,k,j}_3(t,x)\|_R^2\\
&\qquad \qquad + M\|R\|^2_{\mathcal{L}(\X)}(3 e^{\zeta_R t} M_2 K_1(t,\zeta_R)+\|R\|_{\mathcal{L}(\X)}e^{3\zeta_R t})\|h\|_R\|k\|_R\|j\|_R\|\delta^{h,k,j}_3(t,x)\|_R,
\end{align*}
where $M$, $M_2$ and $K_1(t,\zeta_R)$ are the objects appearing in Hypothesis \ref{STR1}\eqref{STR1.3}, \eqref{2R} and \eqref{defn_K1}, respectively. Recalling that for every $x\in\X$ and $h,k,j\in H_R$, by \cite[Proposition A.1.3]{CER1}, the map $t\mapsto \|\delta^{h,k,j}_3(t,x)\|_R$ is differentiable, we can divide \eqref{op} by $\|\delta^{h,k,j}_3(t,x)\|_R$ to get
\begin{align*}
\frac{d}{dt} \|\delta^{h,k,j}_3(t,x)\|_R\leq &\zeta_R\|\delta^{h,k,j}_3(t,x)\|_R\notag\\
&+M\|R\|^2_{\mathcal{L}(\X)}(3 e^{\zeta_R t} M_2 K_1(t,\zeta_R)+\|R\|_{\mathcal{L}(\X)}e^{3\zeta_R t})\|h\|_R\|k\|_R\|j\|_R.
\end{align*}
By a standard comparison result we obtain \eqref{2R} with $M_3:=M\max\{3M_2 ,\|R\|_{\mathcal{L}(\X)}\}\|R\|_{\mathcal{L}(\X)}^2$.
\end{proof}

We left the complex dependence on $t$ in the constants $K_1$ and $K_2$, defined in \eqref{defn_K1} and \eqref{defn_K2}, for the sake of completeness. For the rest of the paper we will just need the following simpler estimate. The proof is standard and it is left to the reader.
\begin{cor}\label{cappone}
Assume that Hypotheses \ref{STR2} hold true. If $\zeta_R\neq 0$, then for any $t>0$ 
\begin{align*}
\max\{K_1(t,\zeta_R), K_2(t,\zeta_R)\}\leq \frac{1+\abs{\zeta_R}}{\zeta_R^2}\max\{e^{t\zeta_R},e^{3t\zeta_R}\};
\end{align*}
where $K_1$ and $K_2$ are introduced in \eqref{defn_K1} and \eqref{defn_K2}, respectively.
\end{cor}

Now we need to introduce a Gateaux type derivation along $H_R$. This notion of derivability was already considered in \cite{BF22} and it will be fundamental to prove a chain rule type result (Corollary \ref{accatena}).

\begin{defi}\label{defn_Gateaux}
Assume Hypothesis \ref{STR1}\eqref{STR1.1} holds true. We say that a function $\Phi:\X\ra\X$ is $H_R$-Gateaux differentiable if for every $x\in\X$ and $h\in H_R$ there exists $\eps_{x,h}>0$ such that the function $\varphi_{x,h}:(-\eps_{x,h},\eps_{x,h})\ra\X$ defined as 
\[\varphi_{x,h}(r):=\Phi(x+rh)-\Phi(x),\qquad x\in\X,\ h\in H_R,\ r\in(-\eps_{x,h},\eps_{x,h});\]
is $H_R$-valued and there exists $T_x\in \mathcal{L}(H_R)$ such that for every $h\in H_R$
\begin{align*}
\lim_{r\ra 0}\norm{\frac{1}{r}\varphi_{x,h}(r)-T_xh}_R=0.
\end{align*}
$T_x$ is called $H_R$-Gateaux derivative of $\Phi$ at the point $x\in \X$ and we denote it by $\J_{G,R}\Phi(x)$. For any $k\in\N$, in a similar way we can define a $k$-times $H_R$-Gateaux differentiable map, and we denote by $\J^k_{G,R}\Phi(x)$ the $H_R$-Gateaux derivative of $\Phi$ of order $k$ at $x\in\X$. We remark that $\J^k_{G,R}\Phi(x)$ belongs to $\mathcal{L}^{(k)}(H_R)$.
\end{defi}

The derivation introduced in Definition \ref{defn_Gateaux} is the correct type of derivation that we will need in the following result.

\begin{prop}\label{Hregbis}
Assume that Hypotheses \ref{STR2} hold true. For any $t>0$ the map $x\mapsto X(t,x)$ is three times $H_R$-Gateaux differentiable $\mathbb{P}$-a.e. and its first, second and third Gateaux derivative operators are $\J_GX(t,x)$, $\J^2_GX(t,x)$ and $\J^3_GX(t,x)$, respectively. Moreover for any $i=1,2,3$ and $t>0$ the map $\J^i_GX(t,\cdot):\X\ra \mathcal{L}^{(i)}(H_R)$ is uniformly continuous, almost everywhere with respect to $\mathbb{P}$.
\end{prop}

\begin{proof}
All estimates in this proof are meant almost everywhere with respect to $\mathbb{P}$. We prove the statements for the first $H_R$-Gateaux derivative, the proof for the second and the third $H_R$-Gateaux derivative operator are similar. We start by proving that, for any $t>0$, the map $x\mapsto X(t,x)$ is $H_R$-Gateaux differentiable $\mathbb{P}$-a.e. and its $H_R$-Gateaux derivative at $x\in\X$ in the direction $h\in H_R$ is $\J_GX(t,x)h$. We recall that by Hypotheses \ref{STR2} and \cite[Proposition I.5.5]{EN00} there exist $B\geq 1$ and $\theta\in\R$ such that
\begin{align}\label{prop_nagel}
\|e^{tA}\|_{\mathcal{L}(H_R)}\leq B e^{\theta t},\qquad t\geq 0.
\end{align} 
So by \eqref{sol_mild}, \eqref{eqdelta1} and \eqref{prop_nagel} it holds
\begin{align*}
\bigg\|\frac{X(t,x+rh)-X(t,x)}{r}&-\J_G X(t,x)h\bigg\|_R\\
&\!\!\!\!\!\!\!\!\!\!\!\!\!\!\!\!\!\!\!\!\!\!\!\!\leq B\int^t_0 e^{\theta s}\norm{\dfrac{G(X(s,x+rh))-G(X(s,x))}{r}-\J G(X(s,x))\J_G X(s,x)h}ds.
\end{align*}
Hence, by the fact that, for every $t>0$, the map $x\mapsto X(t,x)$ is Gateaux differentiable and the fact that $G$ is a three times Fr\'echet differentiable function with bounded and continuous derivative operators, we can apply the dominated convergence theorem to obtain for any $x\in\X$ and $h\in H_R$
\[
\lim_{r\ra 0}\norm{\dfrac{X(t,x+rh)-X(t,x)}{r}-\J_G X(t,x)h}_R=0.
\]
So, for any $t>0$, the map $x\mapsto X(t,x)$ is $H_R$-Gateaux differentiable (in the sense of Definition \ref{defn_Gateaux}) and its $H_R$-Gateaux derivative operator is $\J_G X(t,x)$.

Now we show that, for any $t>0$, the map $x\mapsto \D_GX(t,x)$ is uniformly continuous. By Hypothesis \ref{STR1}\eqref{STR1.3} the function $\D G:\X\ra\mathcal{L}(\X)$ is uniformly continuous, namely there exists a continuous and increasing function $\omega:[0,+\infty)\ra[0,+\infty)$ such that $\lim_{s\ra 0^+}\omega(s)=0$ and 
\begin{align*}
\|\D G(x_1)-\D G(x_2)\|_{\mathcal{L}(\X)}\leq \omega(\|x_1-x_2\|),\qquad x_1,x_2\in\X.
\end{align*}
We claim that, for every $T>0$ and $t\in[0,T]$, the map $x\mapsto \D G(X(t,x))$ is uniformly continuous with a modulus of continuity independent of $t$. Indeed, by \eqref{Lip_mild}, for any $x_1,x_2\in\X$
\begin{align}
\sup_{t\in[0,T]}\|\D G(X(t,x_1))-\D G(X(t,x_2))\|_{\mathcal{L}(\X)}&\leq \sup_{t\in[0,T]}\omega(\|X(t,x_1)-X(t,x_2)\|)\notag\\
&\leq \omega(\eta \|x_1-x_2\|),\label{unif_cont_DGX}
\end{align}
where $\eta$ is the constant appearing in \eqref{Lip_mild}. For simplicity sake we set $\delta_1^h(t,x):=\D_GX(t,x)h$ for every $t>0$, $h\in H_R$ and $x\in\X$. Let $h\in H_R$, $x_1,x_2\in\X$ and $T>0$. By Hypotheses \ref{STR2}, Proposition \ref{accar} and \eqref{eqdelta1} for every $t\in [0,T]$ we have
\begin{align}
\|\delta^h_1(t,x_1)&-\delta^h_1(t,x_2)\|_R\notag\\
&\leq \int^t_0\|e^{sA}\|_{\mathcal{L}(H_R)}\|\J G(X(s,x_1))\delta^h_1(s,x_1)-\J G(X(s,x_2))\delta^h_1(s,x_2)\|ds\notag\\
&\leq \int^t_0\|e^{sA}\|_{\mathcal{L}(H_R)}\|\J G(X(s,x_1))-\J G(X(s,x_2))\|_{\mathcal{L}(\X)}\|\delta^h_1(s,x_1)\|ds\notag\\
&\qquad+\int^t_0\|e^{sA}\|_{\mathcal{L}(H_R)}\|\J G(X(s,x_2))\|_{\mathcal{L}(\X)}\|\delta^h_1(s,x_1)-\delta^h_1(s,x_2)\|ds\notag\\
&=:I_1+I_2.\label{Striscia1}
\end{align}
By Proposition \ref{accar}, \eqref{1R}, \eqref{prop_nagel} and \eqref{unif_cont_DGX} it holds
\begin{align}
I_1 
&\leq B\|R\|_{\mathcal{L}(\X)}\|h\|_R \omega(\eta\|x_1-x_2\|)\int_0^T e^{(\theta+\zeta_R) s} ds\notag\\
&=\frac{B\|R\|_{\mathcal{L}(\X)}(e^{(\theta+\zeta_R)T}-1)}{\theta+\zeta_R}\|h\|_R \omega(\eta\|x_1-x_2\|).\label{Striscia2}
\end{align}
By Hypothesis \ref{STR1}\eqref{STR1.3}, Proposition \ref{accar} and \eqref{prop_nagel} it holds
\begin{align}
I_2 
&\leq BM\|R\|_{\mathcal{L}(\X)}\int_0^T e^{\theta s} \|\delta^h_1(s,x_1)-\delta^h_1(s,x_2)\|_Rds. \label{Striscia3}
\end{align}
Combining \eqref{Striscia1}, \eqref{Striscia2} and \eqref{Striscia3} we get
\begin{align*}
\|\delta^h_1(t,x_1)-\delta^h_1(t,x_2)\|_R &\leq \frac{B\|R\|_{\mathcal{L}(\X)}(e^{(\theta+\zeta_R)T}-1)}{\theta+\zeta_R}\|h\|_R \omega(\eta\|x_1-x_2\|)\\
&\qquad + BM\|R\|_{\mathcal{L}(\X)}\int_0^T e^{\theta s} \|\delta^h_1(s,x_1)-\delta^h_1(s,x_2)\|_Rds.
\end{align*}
By the Gr\"onwall inequality there exists a positive constant $\Lambda(T)$ independent on $h\in H_R$ and on $x_1,x_2\in\X$ such that
\begin{align}\label{UC}
\|\delta^h_1(t,x_1)-\delta^h_1(t,x_2)\|_R &\leq \Lambda(T)\|h\|_R \omega(\eta\|x_1-x_2\|)
\end{align}
%
Hence by \eqref{UC} we conclude that, for every $t>0$, the function $\J_GX(t,\cdot):\X\ra \mathcal{L}(H_R)$ is uniformly continuous. The proof of the uniform continuity of the second and third $H_R$-Gateaux derivatives are similar.
\end{proof}

We end this section by stating a chain rule type result which will be used later in the paper. The proof is standard and can be obtained following the same arguments used in the proof of \cite[Corollary 21]{BF20}. We give it for the sake of completeness.

\begin{cor}\label{accatena}
Assume that Hypotheses \ref{STR2} hold true. If $g\in C^1_{b}(\X)$, then for any $t> 0$ the map $x\mapsto (g\circ X)(t,x)$ is $H_R$-differentiable. Furthermore for any $h\in H_{R}$, $x\in X$ and $t> 0$
it holds 
\[(\J_{R}(g\circ X))(t,x)h=(\D_{R}g) (X(t,x)) \J_G X(t,x)h.\]
\end{cor}

\begin{proof}
Since $g\in C^1_{b}(\X)$, by Proposition \ref{dalpha}, $g$ is also $H_R$-differentiable, then for every $x\in\X$ and $h\in H_{R}$
\[g(x+\eps h)=g(x)+\eps \D_{R}g(x)h+o(\eps),\qquad \text{ for }\eps\ra 0.   \]
We define for any $x\in\X$, $h\in H_{R}$, $t> 0$ and $\eps>0$
\[K_{\eps}(t,x,h):=X(t,x+\eps h)-X(t,x)-\eps\J_G X(t,x)h.\]
By Proposition \ref{Hreg}, for any $T>0$, $x\in\X$ and $h\in H_R$
\begin{align}\label{LOSTSOUL}
\sup_{s\in [0,T]}\E[\|K_{\eps}(s,x,h)\|^2_{R}]=o(\eps)\qquad \text{ for }\eps\ra 0.
\end{align}
Hence, letting $\eps$ tend to zero, for any $x\in\X$, $h\in H_{R}$ and $t> 0$ it holds
\begin{align*}
g\big(X(t, x+\eps h)\big)& =g\big(X(t,x)+\eps\J_G X(t,x)h+K_{\eps}(t,x,h)\big)\\
& =g\big(X(t,x)+\eps(\J_G X(t,x)h+\eps^{-1}K_{\eps}(t,x,h)\big)\\
& =g(X(t,x))+(\D_{R}g)(X(t,x))(\eps\J_G X(t,x)h+K_{\eps}(t,x,h))+o(\eps).
\end{align*}
So, letting $\eps$ tend to zero, by the boundedness of the $H_R$-derivative operator of $g$ (Proposition \ref{dalpha}) and \eqref{LOSTSOUL} we get, for any $x\in\X$, $h\in H_{R}$ and $T>t>0$ 
\begin{align*}
0&\leq \E\sq{\abs{g\big(X(t,x+\eps h)\big)-g(X(t,x))-\eps(\D_{R}g)(X(t,x))\J_G X(t,x)h}^2}\\
&\leq \sup_{s\in[0,T]}\E\sq{\abs{g\big(X(s,x+\eps h)\big)-g(X(s,x))-\eps(\D_{R}g)(X(s,x))\J_G X(s,x)h}^2}\\
&= \sup_{s\in[0,T]}\E\Big[\abs{(\D_{R}g)(X(s,x))K_{\eps}(s,x,h)}^2\Big]+o(\eps)\\
&\leq \pa{\sup_{y\in\X}\norm{\J_R g(y)}_{\mathcal{L}(H_R;\R)}}\pa{\sup_{s\in[0,T]}\E\Big[\norm{K_{\eps}(s,x,h)}_R^2\Big]}+o(\eps)\\
&=\left(1+\sup_{y\in\X}\norm{\J_R g(y)}_{\mathcal{L}(H_R;\R)}\right)o(\eps)
\end{align*}
This concludes the proof. 
\end{proof}

\section{$H_R$ regularity of the transition semigroup $\{P(t)\}_{t\geq 0}$}\label{Regularity_Semigroup}

In this section we will show that the transition semigroup, introduced in \eqref{intr_trans}, satisfies some regularity properties in the sense of Definition \ref{intr_defn_diff}. The next proposition is a variant of the Bismut--Elworthy--Li formula (see \cite{EL-LI1}) adapted to our purposes.

\begin{thm}\label{BEL}
Assume that Hypotheses \ref{STR2} hold true. If $\varphi\in {\rm BUC}(\X)$, then for any $t>0$ the map $x\mapsto P(t)\varphi(x)$ belongs to ${\rm BUC}^3_R(\X)$. Moreover for any $x\in\X$, $h,k,j\in H_R$ and $t> 0$ it holds
\begin{align}\label{FDeq1}
\D_R P(t)\varphi(x)h &=\frac{1}{t}\mathbb{E}\left[\varphi(X(t,x))\int_0^t\scal{\J_G X(s,x)h}{RdW(s)}_{R}\right];\\
\D^2_R P(t)\varphi(x)(h,k) &=\frac{2}{t}\mathbb{E}\left[\D_R\left( P(t/2)\varphi(X(t/2,x))\right)k\int_0^{t/2}\scal{\J_G X(s,x)h}{RdW(s)}_{R}\right]\notag\\
&\quad +\frac{2}{t}\mathbb{E}\left[P(t/2)\varphi(X(t/2,x))\int_0^{t/2}\scal{\J^2_G X(s,x)(h,k)}{RdW(s)}_{R}\right];\label{FDeq2}\\
\D^3_R P(t)\varphi(x)(h,k,j)&=\frac{2}{t}\mathbb{E}\left[\D^2_R\left( P(t/2)\varphi(X(t/2,x))\right)(k,j)\int_0^{t/2}\scal{\J_G X(s,x)h}{RdW(s)}_{R}\right]\notag\\
&\quad +\frac{2}{t}\mathbb{E}\left[\D_R\left( P(t/2)\varphi(X(t/2,x))\right)k\int_0^{t/2}\scal{\J^2_G X(s,x)(h,j)}{RdW(s)}_{R}\right]\notag\\
&\quad +\frac{2}{t}\mathbb{E}\left[\D_R\left( P(t/2)\varphi(X(t/2,x))\right)j\int_0^{t/2}\scal{\J^2_G X(s,x)(h,k)}{RdW(s)}_{R}\right]\notag\\
&\quad +\frac{2}{t}\mathbb{E}\left[P(t/2)\varphi(X(t/2,x))\int_0^{t/2}\scal{\J^3_G X(s,x)(h,k,j)}{RdW(s)}_{R}\right].\label{FDeq3}
\end{align}
\end{thm}

\begin{proof}
Observe that once we prove that \eqref{FDeq1}, \eqref{FDeq2} and \eqref{FDeq3} hold true, then the uniform continuity follows by \eqref{1R}, \eqref{2R}, \eqref{3R} and Proposition \ref{Zone400} (with $\mathcal{M}=\X$, $d_{\mathcal{M}}=\norm{\cdot}$, $Y=\Omega$, $\mu=\mathbb{P}$ and $Z=\R$). By \cite[Lemma 2.3]{PES-ZA1}, for any $\varphi\in {\rm BUC}^2(\X)$, $t>0$ and $x\in\X$ it holds
\begin{equation}\label{raaaa}
\varphi(X(t,x))=P(t)\varphi(x)+\int_0^t\scal{\nabla P(t-s)\varphi(X(s,x))}{RdW(s)},
\end{equation}
and $P(t)\varphi$ belongs to ${\rm BUC}^2(\X)$. By Proposition \ref{dalpha} $P(t)\varphi$ is $H_R$-differentiable and \eqref{raaaa} becomes for $t>0$ and $x\in\X$
\begin{equation}\label{raaaa2}
\varphi(X(t,x))=P(t)\varphi(x)+\int_0^t\scal{\nabla_R P(t-s)\varphi(X(s,x))}{RdW(s)}_R.
\end{equation}
By Proposition \ref{Hreg}, for any $h\in H_R$, the following process is well defined 
\begin{equation}\label{marti}
\left\lbrace\int_0^t\scal{\J_G X(s,x)h}{RdW(s)}_{R}\right\rbrace_{t\geq 0}
\end{equation}
Multiplying both sides of \eqref{raaaa2} by \eqref{marti} and taking the expectations we get for $t>0$, $x\in\X$ and $h\in H_R$
\begin{align*}
\mathbb{E}\bigg[\varphi(X(t,x))&\int_0^t \scal{\J_G X(s,x)h}{RdW(s)}_{R}\bigg]\\ 
&=\mathbb{E}\sq{P(t)\varphi(x)\int_0^t\scal{\J_G X(s,x)h}{RdW(s)}_{R}}\\
&\qquad +\mathbb{E}\sq{\int_0^t\scal{\nabla_R P(t-s)\varphi(X(s,x))}{RdW(s)}_{R}\int_0^t\scal{\J_G X(s,x)h}{RdW(s)}_{R}}.
\end{align*}
Since $R:\X\ra H_R$ is continuous (Proposition \ref{accar}), then $\{RW(t)\}_{t\geq 0}$ is a $H_R$-cylindrical Wiener process (see \cite[Remark 5.1]{DA-ZA4}).
By Proposition \ref{Hreg} for every $t>0$, $x\in\X$ and $h\in H_R$
\[
\int_0^t\mathbb{E}\big[\Vert \J_G X(s,x)h\Vert_{R}^2\big]ds<+\infty,
\] 
so, by \cite[Remark 2]{EL-LI1}, the process defined in \eqref{marti} is a martingale. Hence for any $t> 0$, $x\in\X$ and $h\in H_R$
\[
\mathbb{E}\sq{P(t)\varphi(x)\int_0^t\scal{\J_G X(s,x)h}{RdW(s)}_{R}}=0.
\]
We recall that since $L^2(\Omega,\mathbb{P})$ is a Hilbert space, then the polarization identity holds true, namely for every $\xi_1,\xi_2\in L^2(\Omega,\mathbb{P})$ it holds
\begin{align}\label{cold}
\E\sq{\xi_1 \xi_2}=\frac{1}{4}\E\big[|\xi_1+\xi_2|^2\big]-\frac{1}{4}\E\big[|\xi_1-\xi_2|^2\big].
\end{align}
Let $\Phi(s):=\D_{R} P(t-s)\varphi(X(s,x))$ and $\Gamma(s):=\J_G X(s,x)h$. Now we apply \eqref{cold} with $\xi_1=\int_0^t\scal{\Phi(s)}{RdW(s)}_{R}$ and $\xi_2=\int_0^t\scal{\Gamma(s)}{RdW(s)}_{R}$ and using the It\^o isometry we get
\begin{align}
\mathbb{E}\bigg[\int_0^t\scal{\Phi}{RdW}_{R}&\int_0^t\scal{\Gamma}{RdW}_{R}\bigg]\notag\\
=\frac{1}{4}&\mathbb{E}\left[\left(\int_0^t\scal{\Phi+\Gamma}{RdW}_{R}\right)^2\right]-\frac{1}{4}\mathbb{E}\left[\left(\int_0^t\scal{\Phi-\Gamma}{RdW}_{R}\right)^2\right]\notag\\
=\frac{1}{4}&\mathbb{E}\left[\int_0^t\norm{\Phi+\Gamma}_R^2ds\right]-\frac{1}{4}\mathbb{E}\left[\int_0^t\norm{\Phi-\Gamma}_R^2ds\right]=\mathbb{E}\left[\int_0^t\scal{\Phi}{\Gamma}_R ds\right]\label{pola}.
\end{align}
Recalling that $P(t)\varphi$ is Fr\'echet differentiable, then by Proposition \ref{dalpha}, Corollary \ref{accatena} (with $g=P(t-s)\varphi$) and \eqref{pola} we obtain for $t>0$, $x\in\X$ and $h\in H_R$
\begin{align*}
&\mathbb{E}\bigg[\int_0^t\scal{(\nabla_{R} P(t-s)\varphi)(X(s,x))}{RdW(s)}_{R}\int_0^t\scal{\J_G X(s,x)h}{RdW(s)}_{R}\bigg]\\
&=\mathbb{E}\sq{\int_0^t\scal{(\nabla_{R} P(t-s)\varphi)(X(s,x))}{\J_G X(s,x)h}_{R}ds}\\
&=\mathbb{E}\bigg[\int_0^t \J_R(((P(t-s)\varphi)\circ X)(s,x))h ds\bigg]=\int_0^t \J_{R}\mathbb{E}\big[(P(t-s)\varphi\circ X)(s,x)\big]hds.
\end{align*}
By the very definition of $P(t)$ we know that $\mathbb{E}[(P(t-s)\varphi\circ X)(s,x)]=(P(s)P(t-s)\varphi)(x)=P(t)\varphi(x)$. So we conclude for $t>0$, $x\in\X$ and $h\in H_R$
\begin{align*}
\mathbb{E}\sq{\varphi(X(t,x))\int_0^t\scal{\J_G X(s,x)h}{RdW(s)}_{R}} &=\int_0^t\D_{R} P(t)\varphi(x)hds=t\D_{R} P(t)\varphi(x)h.
\end{align*}
We have proved \eqref{FDeq1} for any $\varphi\in {\rm BUC}^2(\X)$. Using the same approximation arguments of the proof of \cite[Proposition 4.4.3]{CER1} it is possible to show that \eqref{FDeq1} is verified for any $\varphi\in {\rm BUC}(\X)$. So by Proposition \ref{Hregbis} and \eqref{FDeq1}, for any $t>0$ and $\varphi\in {\rm BUC}(\X)$ the function $P(t)\varphi$ belongs to ${\rm BUC}^1_R(\X)$. To prove \eqref{FDeq2} observe that for every $x\in\X$, $h,k\in H_R$ and $t> 0$ it holds
\begin{align*}
\D^2_R P(t)\varphi(x)(h,k)=\D_R\left(\D_R P(t)\varphi(x)h\right)k=\D_R\left(\D_R P(t/2)(P(t/2)\varphi(x))h\right)k.
\end{align*}
So by \eqref{FDeq1}
\begin{align}
\D^2_R P(t)\varphi(x)(h,k)&=\J_R\left(\frac{2}{t}\mathbb{E}\left[\left( P(t/2)\varphi(X(t/2,x))\right)\int_0^{t/2}\scal{\J_G X(s,x)h}{RdW(s)}_{R}\right]\right)k\notag\\
&=\frac{2}{t}\mathbb{E}\left[\J_R\left(\left( P(t/2)\varphi(X(t/2,x))\right)\int_0^{t/2}\scal{\J_G X(s,x)h}{RdW(s)}_{R}\right)k\right].\label{speri}
\end{align}
So, differentiating the product in \eqref{speri}, we obtain \eqref{FDeq2}. By Proposition \ref{Hregbis} and \eqref{FDeq2}, for any $\varphi\in {\rm BUC}(\X)$ the function $P(t)\varphi$ belongs to ${\rm BUC}^2_R(\X)$. \eqref{FDeq3} follows by the same arguments.
\end{proof}

Theorem \ref{BEL} allows us to prove some estimates that are useful in general and extremely important for our results. 

\begin{prop}\label{boss}
Assume that Hypotheses \ref{STR2} hold true. There exists a positive constant $K$, depending only on $M_2$, $M_3$ and $\zeta_R$ (the constants appearing in Proposition \ref{Hreg} and in Hypotheses \ref{STR2}), such that for any $i=1,2,3$, $\varphi\in {\rm BUC}(\X)$, $t> 0$ and $x\in\X$ it holds
\begin{align*}
\|\D^i_R P(t)\varphi(x)\|_{\mathcal{L}^{(i)}(H_R;\R)} \leq \frac{K}{t^{i/2}}\max\left\{1,\frac{t^{i-1}}{1+\abs{\zeta_R}t^{i-1}},e^{3t\zeta_R}\right\}\Vert\varphi\Vert_{\infty}.
\end{align*}
In particular for any $i=1,2,3$, $\varphi\in {\rm BUC}(\X)$, $t> 0$ and $x\in\X$ it holds
\begin{align}\label{boss1}
\|\D^i_R P(t)\varphi(x)\|_{\mathcal{L}^{(i)}(H_R;\R)} \leq K\frac{e^{3t|\zeta_R|}}{t^{i/2}}\Vert\varphi\Vert_{\infty}.
\end{align}
Furthermore, for any $\alpha\in(0,1)$, $t>0$ and $\varphi\in {\rm BUC}(\X)$ it holds
\begin{align}
[\D_R P(t)\varphi]_{{\rm BUC}_R^\alpha(\X;\mathcal{L}(H_R;\R))} &\leq 2^{1-\alpha}K\frac{e^{3t|\zeta_R|}}{t^{(1+\alpha)/2}}\Vert\varphi\Vert_{\infty}.\label{Resident3}
\end{align}
\end{prop}

\begin{proof}
We begin to prove the statements in the case $\zeta_R\neq 0$. By Proposition \ref{Hreg} and Corollary \ref{cappone} for any $t>0$, $x\in\X$ and $h,k,j\in H_R$
\begin{align}
\mathbb{E}\left[\int^t_0\|\D_GX(s,x)h\|^2_Rds\right]&\leq c(t)\norm{h}^2_R;\label{1}\\
\mathbb{E}\left[\int^t_0\|\D^2_GX(s,x)(h,k)\|^2_Rds\right]&\leq c(t)\norm{h}^2_R\norm{k}^2_R;\label{2}\\
\mathbb{E}\left[\int^t_0\|\D^3_GX(s,x)(h,k,j)\|^2_Rds\right]& \leq c(t)\norm{h}^2_R\norm{k}^2_R\norm{j}^2_R.\label{3}
\end{align}
where
\begin{align*}
c(t)=\eqsys{\max\left\{1,M_2,3M^2_2,M_3\right\}\max\{1,(1+\abs{\zeta_R})^2\abs{\zeta_R}^{-5}\}(1-e^{2t\zeta_R}),& \zeta_R<0;\\
\max\left\{1,M_2,3M^2_2,M_3\right\}\max\{1,(1+\abs{\zeta_R})^2\abs{\zeta_R}^{-5}\}(e^{6t\zeta_R}-1),  &\zeta_R>0.}
\end{align*}
Hence there exists a constant $C:=C(M_2,M_3,\zeta_R)>0$ such that, for any $t>0$
\begin{align}
c(t)\leq C(t\chi_{(0,1)}+\max\{1,e^{6\zeta_R t}\}\chi_{[1,+\infty)}),\label{m1}
\end{align}
where $\chi_{(0,1)}$ and $\chi_{[1,+\infty)}$ denote the characteristic functions of the sets $(0,1)$ and $[1,+\infty)$, respectively.
We fix $\varphi\in {\rm BUC}(\X)$, $t> 0$, $x\in\X$ and $h,k,j\in H_R$. By \eqref{FDeq1}, \eqref{1} and the It\^o isometry we get
\begin{align}\label{uno}
\vert \D_R P(t)\varphi(x)h\vert^2 \leq \frac{c(t)}{t^2}\norm{h}^2_R\Vert\varphi\Vert^2_{\infty}.
\end{align}
By \eqref{FDeq2}, \eqref{1}, \eqref{2}, \eqref{uno}, the It\^o isometry, the Jensen inequality and the contractivity of $P(t)$ in ${\rm BUC}(\X)$ we obtain
\begin{align}\label{due}
\vert \D^2_R P(t)\varphi(x)(h,k)\vert^2&\leq \frac{8}{t^4}\big(4c^2(t/2)+c(t/2)t^2\big)\norm{h}^2_R\norm{k}^2_R\Vert\varphi\Vert^2_{\infty}.
\end{align}
By \eqref{FDeq3}, \eqref{1}, \eqref{2}, \eqref{3}, \eqref{uno}, \eqref{due}, the It\^o isometry, the Jensen inequality and the contractivity of $P(t)$ in ${\rm BUC}(\X)$ we obtain
\begin{align}\label{tre}
\vert \D^3_R P(t)\varphi &(x)(h,k,j)\vert^2\notag\\
&\leq\frac{64}{t^6}\Big(8\big(4c^2(t/2)+c(t/2)t^2\big)+8c^2(t/2)t^2+c(t/2)t^4\Big)\norm{h}^2_R\norm{k}^2_R\norm{j}^2_R\Vert\varphi\Vert^2_{\infty}
\end{align}
By \eqref{uno}, \eqref{due}, \eqref{tre} and \eqref{m1} there exists $\tilde{K}:=\tilde{K}(M_2,M_3,\zeta_R)>0$ such that for every $x\in\X$ and $t>0$
\begin{align}\label{unouno}
\|\D^i_R P(t)\varphi(x)\|_{\mathcal{L}^{(i)}(H_R;\R)} \leq \frac{\tilde{K}}{t^{i/2}}\max\{1,e^{3t\zeta_R}\}\Vert\varphi\Vert_{\infty}.
\end{align}

Now we prove the statements in the case $\zeta_R=0$. By Proposition \ref{Hreg} there exists a positive constant $C:=C(M_2,M_3)$ such that for any $x\in\X$ and $h,k,j\in H_R$ we have
\begin{align}
\mathbb{E}\left[\int^t_0\|\D_GX(s,x)h\|^2_Rds\right]&\leq Ct\norm{h}^2_R;\label{01}\\
\mathbb{E}\left[\int^t_0\|\D^2_GX(s,x)(h,k)\|^2_Rds\right]&\leq Ct^2\norm{h}^2_R\norm{k}^2_R;\label{02}\\
\mathbb{E}\left[\int^t_0\|\D^3_GX(s,x)(h,k,j)\|^2_Rds\right]& \leq C(t^2+t^3)\norm{h}^2_R\norm{k}^2_R\norm{j}^2_R\label{03}.
\end{align}
We fix $\varphi\in {\rm BUC}(\X)$, $t> 0$, $x\in\X$ and $h,k,j\in H_R$. By \eqref{FDeq1}, \eqref{01} and the It\^o isometry we get
\begin{align}\label{uno0}
\vert \D_R P(t)\varphi(x)h\vert^2 \leq \frac{C}{t}\norm{h}^2_R\Vert\varphi\Vert^2_{\infty}.
\end{align}
By \eqref{FDeq2}, \eqref{01}, \eqref{02}, \eqref{uno0} the It\^o isometry, the Jensen inequality and the contractivity of $P(t)$ in ${\rm BUC}(\X)$ we obtain
\begin{align}\label{due0}
\vert \D^2_R P(t)\varphi(x)(h,k)\vert^2&\leq \frac{8}{t^2}\big(C^2+4Ct^2)\norm{h}^2_R\norm{k}^2_R\Vert\varphi\Vert^2_{\infty}.
\end{align}
In a similar way by \eqref{FDeq3}, \eqref{01}, \eqref{02}, \eqref{03}, \eqref{uno0}, \eqref{due0}, the It\^o isometry, the Jensen inequality and the contractivity of $P(t)$ in ${\rm BUC}(\X)$ we obtain
\begin{align}\label{tre0}
\vert \D^3_R P(t)\varphi(x)(h,k,j)\vert^2\leq\frac{64}{t^2}\Big(4C\frac{1}{t}\big(C^2+4Ct^2)+C^2t+C(t^2+t^3)\Big)\norm{h}^2_R\norm{k}^2_R\norm{j}^2_R\Vert\varphi\Vert^2_{\infty}.
\end{align}
By \eqref{uno0}, \eqref{due0} and \eqref{tre0} there exists a positive constant $\ol{K}:=\ol{K}(M_2,M_3)$ such that
\begin{align}\label{duedue}
\|\D^i_R P(t)\varphi(x)\|_{\mathcal{L}^{(i)}(H_R;\R)} \leq \frac{\ol{K}}{t^{i/2}}\max\{1,t^{i-1}\}\Vert\varphi\Vert_{\infty}.
\end{align}
Combining  \eqref{unouno} and \eqref{duedue} we obtain \eqref{boss1}.

To prove \eqref{Resident3} observe that for any $x\in\X$, $h\in H_R$, $\varphi\in{\rm BUC}(\X)$ and $t>0$, by \eqref{boss1} (with $i=1$), we get
\begin{align}\label{pinky}
\|\D_RP(t)\varphi(x+h)-\D_RP(t)\varphi(x)\|_{\mathcal{L}(H_R;\R)}\leq 2K\frac{e^{3t|\zeta_R|}}{t^{1/2}}\Vert\varphi\Vert_{\infty}.
\end{align}
While for any $x\in\X$, $h\in H_R$, $\varphi\in{\rm BUC}(\X)$ and $t>0$, by \eqref{boss1} (with $i=2$), we get
\begin{align}
\|\D_RP(t)\varphi(x+h)-\D_RP(t)\varphi(x)\|_{\mathcal{L}(H_R;\R)}&=\norm{\int_0^1\D^2_R P(t)\varphi(x+\sigma h)(h,\cdot)d\sigma}_{\mathcal{L}(H_R;\R)}\notag\\
&\leq K\frac{e^{3t|\zeta_R|}}{t}\|h\|_{R}\Vert\varphi\Vert_{\infty}.\label{pinky2}
\end{align}
Therefore for any $\alpha\in(0,1)$, $x\in\X$, $h\in H_R$, $\varphi\in{\rm BUC}(\X)$ and $t>0$, by \eqref{pinky} and \eqref{pinky2} we have
\begin{align*}
\|\D_RP(t)\varphi(x+h)-\D_RP(t)\varphi(x)\|_{\mathcal{L}(H_R;\R)}&\leq \pa{2K\frac{e^{3t|\zeta_R|}}{t^{1/2}}\Vert\varphi\Vert_{\infty}}^{1-\alpha}\pa{K\frac{e^{3t|\zeta_R|}}{t}\|h\|_{R}\Vert\varphi\Vert_{\infty}}^\alpha.
\end{align*}
This conclude the proof.
\end{proof}

For the proof of Theorem \ref{THM_SCHAUDER} we also need a relationship between the derivatives of $\varphi$ and the derivatives of $P(t)\varphi$.
\begin{prop}\label{BEL2}
Assume that Hypotheses \ref{STR2} hold true. For any $\varphi\in \mathfrak{X}$, $x\in\X$, $h,k,j\in H_R$ and $t> 0$
\begin{align}\label{FDeq1D}
\D_R P(t)\varphi(x)h &=\E\left[\scal{\D_R\varphi(x)}{\D_G X(t,x)h}_R\right];\\
\D^2_R P(t)\varphi(x)(h,k)&=\frac{1}{t}\mathbb{E}\left[\D_R\varphi(X(t,x))k\int_0^{t}\scal{\J_G X(s,x)h}{RdW(s)}_{R}\right]\notag\\
&\qquad+\frac{1}{t}\mathbb{E}\left[\varphi(X(t,x))\int_0^t\scal{\J^2_G X(s,x)(h,k)}{RdW(s)}_{R}\right].\label{FDeq2D}
\end{align}
\end{prop}
\begin{proof}
Differentiating under the integral sign \eqref{intr_trans} and \eqref{FDeq1}, by Corollary \ref{accatena}, we obtain \eqref{FDeq1D} and \eqref{FDeq2D}, respectively.
\end{proof}

If $\varphi$ belongs to $\mathfrak{X}$, thanks to Proposition \ref{BEL2}, it is possible to prove results similar to the ones contained in Proposition \ref{boss}, replacing the space ${\rm BUC}(\X)$ with the space $\mathfrak{X}$.

\begin{prop}\label{bossD}
Assume that Hypotheses \ref{STR2} hold true. There exists a positive constant $K'$, depending only on $M_2$, $M_3$ and $\zeta_R$ (the constants appearing in Proposition \ref{Hreg} and in Hypotheses \ref{STR2}), such that for any $i=1,2,3$, $\varphi\in \mathfrak{X}$, $t>0$ and $x\in\X$ 
\begin{align*}
\|\D^i_R P(t)\varphi(x)\|_{\mathcal{L}^{(i)}(H_R;\R)} \leq \frac{K'}{t^{(i-1)/2}}\max\left\{e^{\abs{(2-i)(3-i)/2}t\zeta_R},\frac{t^{i-1}}{1+\abs{\zeta_R}t^{i-1}},e^{4t\zeta_R}\right\}\Vert\varphi\Vert_{\mathfrak{X}}.
\end{align*}
In particular for any $i=1,2,3$, $\varphi\in \mathfrak{X}$, $t>0$ and $x\in\X$ it holds
\begin{align}\label{bossD1}
\|\D^i_R P(t)\varphi(x)\|_{\mathcal{L}^{(i)}(H_R;\R)} \leq K'\frac{e^{4t|\zeta_R|}}{t^{(i-1)/2}}\Vert\varphi\Vert_{\mathfrak{X}}.%
\end{align}
\end{prop}

\begin{proof}
We reprise the same notations of the proof of Proposition \ref{boss}. Let $\varphi\in \mathfrak{X}$, $t> 0$,  $x\in\X$ and $h,k,j\in H_R$. By \eqref{1R} and \eqref{FDeq1D} we get
\begin{align}\label{unoD}
\vert \D_R P(t)\varphi(x)h\vert^2 \leq e^{2\zeta_R t}\norm{h}^2_R\norm{\varphi}^2_{\mathfrak{X}}.
\end{align}
Now let $\zeta_R\neq 0$. By \eqref{1}, \eqref{2}, \eqref{FDeq2D}, \eqref{unoD} the It\^o isometry and the Jensen inequality we obtain
\begin{align}\label{dueD}
|\D^2_R P(t)\varphi(x)(h,k)|^2&=\frac{16}{t^2}\left(e^{2\zeta_R t}+1\right)c(t)\norm{h}^2_R\norm{k}^2_R\norm{\varphi}^2_{\mathfrak{X}}.
\end{align}
By \eqref{FDeq3}, \eqref{1}, \eqref{2}, \eqref{3}, \eqref{unoD}, \eqref{dueD}, the It\^o isometry, the Jensen inequality and the contractivity of $P(t)$ in ${\rm BUC}(\X)$ we have
\begin{align}\label{treD}
\vert \D^3_R P(t)\varphi(x)(h,k,j)\vert^2&\leq\frac{64}{t^2}c(t/2)\left(\frac{16}{t^2}\left(e^{\zeta_R t}+1\right)c(t/2)+(2e^{t\zeta_R}+1)\right)\norm{h}^2_R\norm{k}^2_R\norm{j}^2_R\norm{\varphi}^2_{\mathfrak{X}}.
\end{align}
By \eqref{m1}, \eqref{unoD}, \eqref{dueD} and \eqref{treD} there exists a constant $C_1:=C_1(M_2,M_3,\zeta_R)>0$ such that 
\begin{align}\label{preuno}
\|\D^i_R P(t)\varphi(x)\|_{\mathcal{L}^{(i)}(H_R;\R)} \leq \frac{C_1}{t^{(i-1)/2}}\max\left\{e^{\abs{(2-i)(3-i)/2}t\zeta_R},e^{4t\zeta_R}\right\}\Vert\varphi\Vert_{\mathfrak{X}}.
\end{align}

Now assume that $\zeta_R=0$. By \eqref{01}, \eqref{02}, \eqref{FDeq2D}, \eqref{unoD}, the It\^o isometry and the Jensen inequality we obtain
\begin{align}\label{dueD0}
|\D^2_R P(t)\varphi(x)(h,k)|^2&\leq\frac{16}{t^2}C\left(t+t^2\right)\norm{h}^2_R\norm{k}^2_R\norm{\varphi}^2_{\mathfrak{X}}
\end{align}
By \eqref{FDeq3}, \eqref{01}, \eqref{02}, \eqref{03}, \eqref{unoD}, \eqref{dueD0}, the It\^o isometry, the Jensen inequality and the contractivity of $P(t)$ in ${\rm BUC}(\X)$ we have
\begin{align}\label{treD0}
\vert \D^3_R P(t)\varphi(x)(h,k,j)\vert^2&\leq\frac{64}{t^2}\left(\frac{16}{t^2}C\left(t+t^2\right)Ct+Ct^2+C(t^2+t^3)\right)\norm{h}^2_R\norm{k}^2_R\norm{j}^2_R\norm{\varphi}^2_{\mathfrak{X}}.
\end{align}
So by \eqref{unoD}, \eqref{dueD0}, \eqref{treD0} there exists a positive constant $C_2:=C_2(M_2,M_3)$ such that 
\begin{align}\label{predue}
\|\D^i_R P(t)\varphi(x)\|_{\mathcal{L}^{(i)}(H_R;\R)} \leq \frac{C_2}{t^{(i-1)/2}}\max\left\{1,t^{i-1}\right\}\Vert\varphi\Vert_{\mathfrak{X}}.
\end{align}
Combining  \eqref{preuno} and \eqref{predue} we obtain \eqref{bossD1}.
\end{proof}

By Theorem \ref{intDS}, Proposition \ref{boss}, Proposition \ref{bossD} and Theorem \ref{classic} (with $\K_0,\K_1,\mathcal{H}_1={\rm BUC}(\X)$, $\mathcal{H}_0=\mathfrak{X}$, $\vartheta=\alpha$ and $T=\J_R^i P(t)$ for $i=1,2,3$) we obtain the following result.

\begin{prop}\label{Holder}
Assume that Hypotheses \ref{STR2} hold true. For every $\alpha\in(0,1)$ there exists a positive constant $K_\alpha$, depending only on $M_2$, $M_3$, $\zeta_R$ and $\alpha$ (the constants appearing in Proposition \ref{Hreg} and in Hypotheses \ref{STR2}), such that for any $\varphi\in {\rm BUC}^\alpha_R(\X)$, $i=1,2,3$, $x\in\X$ and $t>0$
\begin{align*}
\|\D^i_R P(t)&\varphi(x)\|_{\mathcal{L}^{(i)}(H_R;\R)}\\
& \leq \frac{K_\alpha}{t^{(i-\alpha)/2}}\max\left\{e^{(1-\alpha)\abs{(2-i)(3-i)/2}t\zeta_R},\frac{t^{i-1}}{1+\abs{\zeta_R}t^{i-1}},e^{4t\zeta_R}\right\}\Vert\varphi\Vert_{{\rm BUC}^\alpha_R(\X)}.
\end{align*}
In particular for any $\varphi\in {\rm BUC}^\alpha_R(\X)$, $i=1,2,3$, $x\in\X$ and $t>0$ it holds
\begin{align}\label{Zone300}
\|\D^i_R P(t)\varphi(x)\|_{\mathcal{L}^{(i)}(H_R;\R)} \leq K_\alpha\frac{e^{4t|\zeta_R|}}{t^{(i-\alpha)/2}}\Vert\varphi\Vert_{{\rm BUC}^\alpha_R(\X)}.%
\end{align}
\end{prop}

\section{The main results}\label{Main_Results}

This section is devoted to the proofs of Theorem \ref{THM_SCHAUDER}, Theorem \ref{Thm_Zyg}, Theorem \ref{Thm_evol} and Proposition \ref{SCHvar}. These proofs are inspired by those presented in \cite{CL19}.

\subsection*{The stationary equation} Here we give the proof of Theorem \ref{THM_SCHAUDER}.

\begin{proof}[Proof of Theorem \ref{THM_SCHAUDER}]
We start by proving the statement for $\lambda>4\abs{\zeta_R}$. 
Let $\alpha\in(0,1)$, $f\in {\rm BUC}^\alpha_R(\X)$ and let $u$ be the function defined in \eqref{GS}. For every $t>0$, by Theorem \ref{BEL}, $P(t)f$ is three times $H_R$-differentiable and observe that, by \eqref{Zone300} (with $i=1$), for any $t>0$, $x\in\X$ and $h\in H_R$
\begin{align}
|P(t)f(x+h)-P(t)f(x)&-\D_RP(t)f(x)h|\notag\\
&=\abs{\int_0^1(\D_RP(t)f(x+\sigma h)-\D_RP(t)f(x))hd\sigma}\notag\\
&\leq \frac{2K_\alpha}{t^{(1-\alpha)/2}}e^{4t|\zeta_R|}\|h\|_R\|f\|_{{\rm BUC}_R^\alpha(\X)},\label{load1}
\end{align}
where $K_\alpha$ is the constant appearing in Proposition \ref{Holder}. In a similar way by \eqref{Zone300} (with $i=2$) we get for any $t>0$, $x\in\X$ and $h,k\in H_R$
\begin{align}
|\D_R P(t)f(x+h)k-\D_R P(t)f(x)k &-\D^2_RP(t)f(x)(h,k)|\notag\\
&\leq \frac{2 K_\alpha}{t^{(2-\alpha)/2}}e^{4t|\zeta_R|}\|h\|_R\|k\|_R\|f\|_{{\rm BUC}_R^\alpha(\X)},\label{load2}
\end{align}
where $K_\alpha$ is the constant appearing in Proposition \ref{Holder}. Hence, by \eqref{load1}, \eqref{load2} and the dominated convergence theorem we get that $u$ is two times $H_R$-differentiable and  for every $x\in\X$ and $\lambda>4|\zeta_R|$ it hold
\begin{align*}
\D_R u(x)&=\int_0^{+\infty}e^{-\lambda t}\D_RP(t)f(x)dt;\qquad \D^2_R u(x)=\int_0^{+\infty}e^{-\lambda t}\D^2_RP(t)f(x)dt.
\end{align*}
By the same arguments used in the proof of \eqref{load1} and \eqref{load2}, we obtain 
\begin{align}
\|\D_R u(x)\|_{\mathcal{L}(H_R;\R)}&\leq\frac{K_\alpha\Gamma\pa{(1+\alpha)/2}}{(\lambda-4|\zeta_R|)^{(1+\alpha)/2}}\|f\|_{{\rm BUC}_R^\alpha(\X)};\label{archie}\\
\|\D^2_R u(x)\|_{\mathcal{L}^{(2)}(H_R;\R)}&\leq \frac{K_\alpha\Gamma\pa{\alpha/2}}{(\lambda-4|\zeta_R|)^{\alpha/2}}\|f\|_{{\rm BUC}_R^\alpha(\X)},\notag
\end{align}
where $\Gamma(z):=\int_0^{+\infty}t^{z-1}e^{-t}dt$ is the Gamma function. The fact that $u\in {\rm BUC}_R^2(\X)$ follows by  Theorem \ref{BEL} and Proposition \ref{Zone400} (with $\mathcal{M}=\X$, $d_{\mathcal{M}}=\norm{\cdot}$, $Y=(0,+\infty)$, $\mu$ is the measure $e^{-\lambda t}dt$ and $Z$ is $\R$ or $\mathcal{L}(H_R;\R)$ or $\mathcal{L}^{(2)}(H_R;\R)$ for the uniform continuity of $u$ or $\D_R u$ or $\D^2_R u$, respectively).

It remains to be proven that $\D^2_Ru\in {\rm BUC}_R^\alpha(\X;\mathcal{L}^{(2)}(H_R;\R))$. For $\lambda>0$, $x\in\X$ and $h\in H_R$ we set
\begin{align*}
a(x):=\int_0^{\norm{h}_R^2}e^{-\lambda t} \D_R^2P(t)f(x)dt;\qquad  b(x):=\int_{\norm{h}_R^2}^{+\infty}e^{-\lambda t} \D_R^2P(t)f(x)dt.
\end{align*}
Hence we have
\begin{align}
\|\D_R^2u(x+h)-\D_R^2u(x)\|_{\mathcal{L}^{(2)}(H_R;\R)}&= \norm{\int_0^{+\infty}e^{-\lambda t}(\D_R^2P(t)f(x+h)-\D_R^2P(t)f(x))dt}_{\mathcal{L}^{(2)}(H_R;\R)}\notag\\
&\leq \norm{a(x+h)-a(x)+b(x+h)-b(x)}_{\mathcal{L}^{(2)}(H_R;\R)}\notag\\
&\leq \norm{a(x+h)-a(x)}_{\mathcal{L}^{(2)}(H_R;\R)}+\norm{b(x+h)-b(x)}_{\mathcal{L}^{(2)}(H_R;\R)}\notag\\
&=:I_1+I_2.\label{Plutonia1}
\end{align}
By \eqref{Zone300} (with $i=2$) we obtain
\begin{align}
I_1&\leq \int_0^{\norm{h}_R^2}e^{-\lambda t}\|\D_R^2P(t)f(x+h)-\D_R^2P(t)f(x)\|_{\mathcal{L}^{(2)}(H_R;\R)}dt\notag\\
&\leq 2\int_0^{\norm{h}_R^2}e^{-\lambda t}\sup_{y\in\X}\|\D_R^2P(t)f(y)\|_{\mathcal{L}^{(2)}(H_R;\R)}dt\leq 2K_\alpha\pa{\int_0^{\norm{h}_R^2}\frac{e^{-(\lambda-4|\zeta_R|) t}}{t^{(2-\alpha)/2}}dt}\|f\|_{{\rm BUC}_R^\alpha(\X)}\notag\\
&\leq 2K_\alpha\pa{\int_0^{\norm{h}_R^2}t^{(\alpha-2)/2}dt}\|f\|_{{\rm BUC}_R^\alpha(\X)}=\frac{4K_\alpha}{\alpha}\|h\|_R^\alpha\|f\|_{{\rm BUC}_R^\alpha(\X)},\label{Plutonia2}
\end{align}
where $K_\alpha$ is the constant appearing in Proposition \ref{Holder}. Moreover by \eqref{Zone300} (with $i=3$) we have
\begin{align}
I_2&\leq \int_{\norm{h}_R^2}^{+\infty}e^{-\lambda t}\|\D_R^2P(t)f(x+h)-\D_R^2P(t)f(x)\|_{\mathcal{L}^{(2)}(H_R;\R)}dt\notag\\
&= \int_{\norm{h}_R^2}^{+\infty}e^{-\lambda t}\norm{\int_0^1\D_R^3P(t)f(x+\sigma h)(h,\cdot,\cdot)d\sigma}_{\mathcal{L}^{(2)}(H_R;\R)}dt\notag\\
&\leq \int_{\norm{h}_R^2}^{+\infty}e^{-\lambda t}\sup_{y\in\X}\|\D_R^3P(t)f(y)\|_{\mathcal{L}^{(3)}(H_R;\R)}\|h\|_Rdt\notag\\
&\leq K_\alpha\pa{\int_{\norm{h}_R^2}^{+\infty}\frac{e^{-(\lambda-4|\zeta|_R)t}}{t^{(3-\alpha)/2}}dt}\|h\|_R\|f\|_{{\rm BUC}_R^\alpha(\X)}\notag\\
&\leq K_\alpha\pa{\int_{\norm{h}_R^2}^{+\infty}t^{(\alpha-3)/2}dt}\|h\|_R\|f\|_{{\rm BUC}_R^\alpha(\X)}=\frac{2K_\alpha}{1-\alpha}\|h\|^\alpha_R\|f\|_{{\rm BUC}_R^\alpha(\X)},\label{Plutonia3}
\end{align}
where $K_\alpha$ is the constant appearing in Proposition \ref{Holder}. Combining \eqref{Plutonia1}, \eqref{Plutonia2} and \eqref{Plutonia3} we get that $\D^2_Ru$ belongs to ${\rm BUC}_R^\alpha(\X;\mathcal{L}^{(2)}(H_R;\R))$.

The proof of the case $0<\lambda<4|\zeta_R|$ use the following perturbation argument. Observe that $u$ solves
\begin{align*}
(4|\zeta_R|+1)y-Nu=f-(\lambda-4|\zeta_R|-1)u.
\end{align*}
By \eqref{archie} we know that $u\in {\rm BUC}_R^1(\X)$, consequentely $f-(\lambda-4|\zeta_R|-1)u$ belongs to ${\rm BUC}_R^\alpha(\X)$ and by the above arguments it follows that $u\in{\rm BUC}_R^{2+\alpha}(\X)$ and the esitmate \eqref{THM_SCHAUDER_STIMA} holds true.
\end{proof}

\subsection*{The case $f\in{\rm BUC}(\X)$}
We proceed to analyze the case $\alpha=0$. As announced in the introduction we need to introduce an appropriate Zygmund space.

\begin{defi}
Assume Hypothesis \ref{STR1}\eqref{STR1.1} holds true and let $Y$ be a Banach space with norm $\norm{\cdot}_Y$. The space $\mathcal{Z}_R(\X;Y)$ is the subspace of ${\rm BUC}(\X;Y)$ consisting of functions $F:\X\ra Y$ such that
\begin{align*}
[F]_{\mathcal{Z}_R(\X;Y)}:=\sup_{\substack{x\in\X\\ h\in H_R\setminus\set{0}}}\frac{\norm{F(x+2h)-2F(x+h)+F(x)}_Y}{\norm{h}_R},
\end{align*}
is finite.
\end{defi}

\noindent The space $\mathcal{Z}_R(\X;Y)$ is a Banach space if endowed with the norm
\[\norm{F}_{\mathcal{Z}_R(\X;Y)}:=\norm{F}_\infty+[F]_{\mathcal{Z}_R(\X;Y)}.\]
It is easy to see that every $H_R$-Lipschitz function belongs to $\mathcal{Z}_R(\X;Y)$, but, even when $\X=Y=H_R=\R$, there are bounded and continuous functions not belonging to $\mathcal{Z}_R(\X;Y)$ (see, for example, \cite{Tri95}).

\begin{thm}\label{Thm_Zyg}
Assume Hypotheses \ref{STR2} hold true. For any $\lambda>0$ and $f\in {\rm BUC}(\X)$ the solution $u$ of \eqref{intr_problema_stazionario}, introduced in \eqref{GS}, belongs to ${\rm BUC}^1(\X)$ and $\D_Ru\in \mathcal{Z}_R(\X;H_R)$. Moreover there exists a positive constant $C$, independent of $f$, such that
\begin{align}\label{estim_Zyg}
\|\D_Ru\|_{\mathcal{Z}_{R}(\X;\mathcal{L}(H_R;\R))}\leq C\norm{f}_\infty.
\end{align}
\end{thm}

\begin{proof}
We prove the statement for $\lambda>3|\zeta_R|$, the general case follows using the same ideas used at the end of the proof of Theorem \ref{THM_SCHAUDER}.
Following the same ideas of the proof of Theorem \ref{THM_SCHAUDER} and using \eqref{Resident3} we get $u\in {\rm BUC}^{1+\alpha}_R(\X)$ for every $\alpha\in(0,1)$. In particular $\D_Ru$ is bounded and uniformly continuous. To prove that $\D_Ru$ belongs to $\mathcal{Z}_R(\X,H_R)$ we consider the functions $a,b:\X\ra\R$ defined, for any $x\in\X$ and $h\in H_R$, as
\begin{align*}
a(x):=\int_0^{\norm{h}_R^2}e^{-\lambda t} \D_RP(t)f(x)dt;\qquad b(x):=\int_{\norm{h}_R^2}^{+\infty}e^{-\lambda t} \D_RP(t)f(x)dt.
\end{align*}
For any $x\in\X$, $h\in H_R$ and $t>0$, by \eqref{boss1} (with $i=1$), we have 
\begin{align}
\|a(x+2h)&-2a(x+h)+a(x)\|_{\mathcal{L}(H_R;\R)}\notag\\
&\leq\int_0^{\|h\|_R^2}e^{-\lambda t}\|\D_RP(t)f(x+2h)-2\D_RP(t)f(x+h)+\D_RP(t)f(x)\|_{\mathcal{L}(H_R;\R)}dt\notag\\
&\leq 4 K\pa{\int_0^{\|h\|^2_R}\frac{e^{-(\lambda-3|\zeta_R|)t}}{t^{1/2}}dt}\|f\|_\infty\leq 4 K\pa{\int_0^{\|h\|^2_R}t^{-1/2}dt}\|f\|_\infty\notag\\
&= 8K\|h\|_R\|f\|_\infty,\label{Freeman}
\end{align}
where $K$ is the constant appearing in Proposition \ref{boss}. Before proceeding we need some intermediate estimates. First of all observe that for every $x\in\X$, $h,k\in H_R$ and $t>0$
\begin{align}
(\D_RP(t)f(x+2h)-\D_RP(t)f(x+h))k &=\int_0^1\D_R^2P(t)f(x+(1+\sigma)h)(h,k)d\sigma;\label{VPN}\\
(\D_RP(t)f(x+h)-\D_RP(t)f(x))k &=\int_0^1\D_R^2P(t)f(x+\sigma h)(h,k)d\sigma.\label{VPN1}
\end{align}
So combining \eqref{VPN} and \eqref{VPN1} and using \eqref{boss1} (with $i=3$),for any $x\in\X$, $h,k\in H_R$ and $t>0$ we get
\begin{align}
\|\D_RP(t)f(x+2h)&-2\D_RP(t)f(x+h)+\D_RP(t)f(x)\|_{\mathcal{L}(H_R;\R)}\notag\\
&=\sup_{\|k\|_R=1}|(\D_RP(t)f(x+2h)-2\D_RP(t)f(x+h)+\D_RP(t)f(x))k|\notag\\
&=\sup_{\|k\|_R=1}\abs{\int_0^1\pa{\D_R^2P(t)f(x+(1+\sigma)h)-\D_R^2P(t)f(x+\sigma h)}(h,k)d\sigma}\notag\\
&=\sup_{\|k\|_R=1}\abs{\int_0^1\int_0^1\D_R^3P(t)f(x+(\tau+\sigma)h)(h,h,k)d\tau d\sigma}\notag\\
&\leq \sup_{y\in\X}\|\D_R^3P(t)f(y)\|_{\mathcal{L}^{(3)}(H_R;\R)}\|h\|_R^2\leq K\frac{e^{3t|\zeta_R|}}{t^{3/2}}\|h\|_R^2\Vert f\Vert_{\infty},\label{roundabout}
\end{align}
where $K$ is the constant introduced in Proposition \ref{boss}. By \eqref{roundabout} for any $x\in\X$, $h,k\in H_R$ and $t>0$ we get
\begin{align}
\|b(x+2h)&-2b(x+h)+b(x)\|_{\mathcal{L}(H_R;\R)}\notag\\
&\leq\int_{\|h\|_R^2}^{+\infty}e^{-\lambda t}\|\D_RP(t)f(x+2h)-2\D_RP(t)f(x+h)+\D_RP(t)f(x)\|_{\mathcal{L}(H_R;\R)} dt\notag\\
&\leq K\pa{\int_{\|h\|_R^2}^{+\infty}\frac{e^{-(\lambda-3|\zeta_R|)t}}{t^{3/2}}dt}\|h\|_R^2\Vert f\Vert_{\infty}\leq K\pa{\int_{\|h\|_R^2}^{+\infty}t^{-3/2}dt}\|h\|_R^2\Vert f\Vert_{\infty}\notag\\
&=2K\|h\|_R\Vert f\Vert_{\infty},\label{Freeman1}
\end{align}
where $K$ is again the constant appearing in Proposition \ref{boss}. Combining \eqref{Freeman} and \eqref{Freeman1} we get \eqref{estim_Zyg}.
\end{proof}

\subsection*{The evolution equation} 
By a procedure similar to the one described in the proofs of Theorem \ref{THM_SCHAUDER} and Theorem \ref{Thm_Zyg} we can obtain similar results for the mild solution $v$ of the evolution equation \eqref{evol_prob}, introduced in \eqref{mild_evol}. To do so we need to introduce another Banach space.

\begin{defi}
Assume Hypothesis \ref{STR1}\eqref{STR1.1} holds true and let $Y$ be a Banach space with norm $\norm{\cdot}_Y$. For any $\alpha\in(0,1)$, $k=0,1,2,\ldots$ and $T>0$ we define ${\rm BUC}_R^{0,k+\alpha}([0,T]\times\X; Y)$ as the set of continuous functions $g:[0,T]\times\X\ra Y$ , that are separately uniformly continuous and such that 
\begin{align*}
\|g\|_{{\rm BUC}_R^{0,k+\alpha}([0,T]\times \X;Y)}:=\sup_{t\in[0,T]}\|g(t,\cdot)\|_{{\rm BUC}_R^{k+\alpha}(\X;Y)},
\end{align*}
if finite. If $Y=\R$ we write ${\rm BUC}_R^{0,k+\alpha}([0,T]\times \X)$.
\end{defi}

\noindent For any $T>0$, $\alpha\in(0,1)$ and $k=0,1,2,\ldots$ the space ${\rm BUC}_R^{0,k+\alpha}([0,T]\times\X; Y)$ is a Banach space if endowed with the norm $\norm{\cdot}_{{\rm BUC}_R^{0,k+\alpha}([0,T]\times \X;Y)}$. Before stating and proving the main result of this subsection we need a preliminary lemma.

\begin{lemm}\label{NFT}
Assume Hypotheses \ref{STR2} hold true. For every $T>0$, $\alpha\in(0,1)$, $k=0,1,2,3$ and $f\in{\rm BUC}_R^{k+\alpha}(\X)$, the map $(t,x)\mapsto P(t)f(x)$ belongs to ${\rm BUC}_R^{0,k+\alpha}([0,T]\times\X)$.
\end{lemm}

\begin{proof}
We just show the case $k=0$, since the other cases follow by similar arguments. Observe that the fact that, for every $t\in[0,T]$, the map $x\mapsto P(t)f(x)$ is uniformly continuous follows by the boundness and uniform continuity of $f$ and Proposition \ref{Zone400} (with $\mathcal{M}=\X$, $d_{\mathcal{M}}=\norm{\cdot}$, $Y=\Omega$, $\mu=\mathbb{P}$ and $Z=\R$), while the uniform continuity of $t\mapsto P(t)f(x)$ follows by its continuity and the compactness of $[0,T]$. To obtain that the map $(t,x)\mapsto P(t)f(x)$ belongs to ${\rm BUC}_R^{0,\alpha}([0,T]\times\X)$ it is enough to note that for any $t>0$
\begin{align}\label{Man1}
\norm{P(t)f}_\infty\leq \norm{f}_\infty,
\end{align}
and that for any $t>0$, $x\in\X$ and $h\in H_R$, by \ref{1R} it holds
\begin{align}
|P(t)f(x+h)-P(t)f(x)|&=\big|\E[f(X(t,x+h))]-\E[f(X(t,x))]\big|\notag\\
&\leq \E\big[\norm{X(t,x+h)-X(t,x)}_{R}^\alpha [f]_{{\rm BUC}_R^\alpha(\X)}\big]\notag\\
&\leq e^{\alpha \zeta_R t}[f]_{{\rm BUC}_R^\alpha(\X)}\norm{h}_R^\alpha.\label{Man2}
\end{align}
Indeed, by \eqref{Man1} and \eqref{Man2}, letting $g(t,x):=P(t)f(x)$ for $t\in[0,T]$ and $x\in\X$ we obtain
\begin{align*}
\|g\|_{{\rm BUC}_R^{0,k+\alpha}([0,T]\times \X)}&=\sup_{t\in [0,T]}(\|P(t)f\|_\infty+[P(t)f]_{{\rm BUC}_R^\alpha(\X)})\leq \max\{1,e^{\alpha\zeta_R T}\}\|f\|_{{\rm BUC}_R^\alpha(\X)}.
\end{align*}
This concludes the proof.
\end{proof}

The following result is in the same spirit as the main results of \cite{KCL75} and \cite{KCL80}.

\begin{thm}\label{Thm_evol}
Assume Hypotheses \ref{STR2} hold true and let $T>0$ and $\alpha\in(0,1)$. If $f\in{\rm BUC}_R^{2+\alpha}(\X)$ and $g\in {\rm BUC}_R^{0,\alpha}([0,T]\times\X)$, then the mild solution $v$ of \eqref{evol_prob}, introduced in \eqref{mild_evol}, belongs to ${\rm BUC}_R^{0,2+\alpha}([0,T]\times\X)$ and there exists a positive constant $C=C(T,\alpha)$, independent of $f$ and $g$, such that
\begin{align*}
\|v\|_{{\rm BUC}_R^{0,2+\alpha}([0,T]\times\X)}\leq C\left(\|f\|_{{\rm BUC}_R^{2+\alpha}(\X)}+\|g\|_{{\rm BUC}_R^{0,\alpha}([0,T]\times\X)}\right).
\end{align*}
\end{thm}

\begin{proof}
We just give a sketch of the proof since it is similar to the proofs of Theorems \ref{THM_SCHAUDER} and \ref{Thm_Zyg}. By Lemma \ref{NFT} the maps $(t,x)\mapsto P(t)f(x)$ belongs to ${\rm BUC}_R^{0,2+\alpha}([0,T]\times\X)$. So we just need to consider the function
\begin{align*}
V(t,x):=\int_0^tP(s)g(t-s,\cdot)(x)ds,\qquad t\in[0,T],\ x\in\X.
\end{align*}
By the same arguments used in the proof of Theorem \ref{THM_SCHAUDER} we obtain that, for every $t\in[0,T]$, the map $x\mapsto V(t,x)$ belongs to ${\rm BUC}_R^2(\X)$, that 
\begin{align*}
\D_R^2V(t,x)=\int_0^t\D_R^2P(s)g(t-s,\cdot)(x)ds,\qquad t\in[0,T],\ x\in\X,
\end{align*}
and that there exists a positive constant $C=C(T,\alpha,\zeta_R)$ such that
\begin{align*}
\sup_{t\in[0,T]}\|V(t,\cdot)\|_{{\rm BUC}_R^2(\X)}\leq C\|g\|_{{\rm BUC}_R^{0,\alpha}([0,T]\times\X)}.
\end{align*}
The uniform continuity of the map $t\mapsto V(t,x)$, for every $x\in\X$, is standard. Finally to prove that the map $x\mapsto\D_R^2V(t,x)$ belongs to ${\rm BUC}_R^\alpha(\X;\mathcal{L}^{(2)}(H_R;\R))$, for every $t\in[0,T]$, we can argue as in the proof of Theorem \ref{THM_SCHAUDER}, introducing the functions $a,b:\X\ra\mathcal{L}^{(2)}(H_R;\R)$ defined, for $h\in H_R$ and $y\in\X$, as
\begin{align*}
a(y):=\int_0^{\min\{t,\|h\|_R^2\}}\D_R^2 P(s)g(t-s,\cdot)(y)ds,\qquad b(y):=\int_{\min\{t,\|h\|_R^2\}}^t\D_R^2 P(s)g(t-s,\cdot)(y)ds.
\end{align*}
Proceeding as in the proof of Theorem \ref{THM_SCHAUDER} we get that there exists a positive constant $C'=C'(T,\alpha,\zeta_R)$ such that
\begin{align*}
\sup_{t\in[0,T]}[\D_R^2V(t,\cdot)]_{{\rm BUC}_R^\alpha(\X;\mathcal{L}^{(2)}(H_R;\R))}\leq C'\sup_{t\in[0,T]}\|g(t,\cdot)\|_{{\rm BUC}_R^\alpha(\X)}.
\end{align*}
This conclude the proof.
\end{proof}

\subsection*{The $H_R$-H\"older perturbation case} Let $A:\Dom(A)\subseteq\X\ra\X$ and $R\in\mathcal{L}(\X)$ such that they satisfy Hypotheses \ref{STR2}. We consider the Ornstein--Uhlenbeck semigroup $\{T(t)\}_{t\geq 0}$ defined by
\[
T(t)\varphi(x):=\int_\X\varphi(e^{tA}x+y)\mathcal{N}(0,Q_t)(dy),\qquad t> 0,\ x\in\X,\ \varphi\in {\rm BUC}(\X);
\]
where, for any $t>0$, we let $Q_t:=\int^t_0e^{sA}R^2e^{sA^*}ds$ and $\mathcal{N}(0,Q_t)$ is the Gaussian measure on $\X$ with mean zero and covariance operator $Q_t$. The semigroup $\{T(t)\}_{t\geq 0}$ is weakly continuous in ${\rm BUC}(\X)$ (see \cite[Appendix B]{CER1}) and its weak generator $L:{\rm Dom}(L)\subseteq {\rm BUC}(\X)\ra{\rm BUC}(\X)$ is the unique closed operator such that
\begin{align*}
R(\lambda,L)f(x)=\int_0^{+\infty} e^{-\lambda s}T(s)f(x)ds,\qquad \lambda>0,\ x\in\X,\ f\in{\rm BUC}(\X).
\end{align*}
Let $\alpha\in (0,1)$ and let $L_\alpha:{\rm Dom}(L_\alpha)\subseteq {\rm BUC}_R^{\alpha}(\X)\ra{\rm BUC}_R^{\alpha}(\X)$ be the part of $L$ in ${\rm BUC}^\alpha_R(\X)$. For $F\in{\rm BUC}^\alpha_R(\X;H_R)$ we consider the stationary equation
\begin{equation}\label{variante}
\lambda u-L_\alpha u-\langle F,\J_Ru\rangle_R=f,\qquad \lambda>0,\ f\in{\rm BUC}_R^\alpha(\X).
\end{equation}
We stress that the following stochastic partial differential equation
\begin{gather*}
\eqsys{
dX(t,x)=\big[AX(t,x)+F(X(t,x))\big]dt+RdW(t), & t>0;\\
X(0,x)=x\in \X,
}
\end{gather*}
may not be well posed and so we do not know if the transition semigroup \eqref{intr_trans} is well defined.
%
We refer to \cite{AD-MA-PR1} for a study of the case when $F=RG$ for some $G\in C_b^\alpha(\X;\X)$ and to \cite{DFPR13} for the case in which $F\in B_b(\X;\X)$ and $A$ is a Laplacian type operator. However it is possible to prove Schauder regularity results for the solution of \eqref{variante} using another technique (see, for example, \cite{DA6}).

\begin{prop}\label{SCHvar}
Let $A$ and $R$ be two operators satisfying Hypotheses \ref{STR2}, let $\alpha\in(0,1)$ and let $f\in {\rm BUC}_R^{\alpha}(\X)$. For any $\lambda>0$ there exists a constant $C_{\alpha,\lambda}>0$ such that if $\|F\|_{{\rm BUC}^\alpha_R(\X;H_R)}\leq C_{\alpha,\lambda}^{-1}$, then \eqref{variante} has a unique solution $u$ belonging to ${\rm BUC}^{2+\alpha}_R(\X)$.
\end{prop} 

\begin{proof}
We start by setting $\psi:=\lambda u-L_\alpha u$. We stress that if $u\in {\rm BUC}_R^{\alpha}(\X)$, then $\psi\in {\rm BUC}_R^{\alpha}(\X)$. Now letting $T:=\langle F,\J_RR(\lambda,L_\alpha)\rangle _R$, then \eqref{variante} reads as
\begin{align}\label{eq1}
\psi-T\psi=f.
\end{align}
By Theorem \ref{THM_SCHAUDER}, whenever $\psi\in {\rm BUC}_R^{\alpha}(\X)$, there exists a constant $C_{\alpha,\lambda}>0$ such that
\[
\|T\psi\|_{{\rm BUC}_R^{\alpha}(\X)}\leq C_{\alpha,\lambda}\|F\|_{{\rm BUC}_R^{\alpha}(\X;H_R)}\|\psi\|_{{\rm BUC}_R^{\alpha}(\X)}.
\]
If $\|F\|_{{\rm BUC}^\alpha_R(\X;H_R)}\leq C^{-1}_{\alpha,\lambda}$, then the map $T:{\rm BUC}_R^{\alpha}(\X)\ra{\rm BUC}_R^{\alpha}(\X)$ is a contraction and so by the contraction mapping theorem \eqref{variante} and \eqref{eq1} have unique solutions $u$ and $\psi$ belonging to ${\rm BUC}_R^{\alpha}(\X)$. Moreover, since $u=R(\lambda,L_\alpha)\psi=R(\lambda,L)\psi$, by Theorem \ref{THM_SCHAUDER}, $u$ belongs to ${\rm BUC}_R^{2+\alpha}(\X)$.
\end{proof}

Clearly the condition on the ${\rm BUC}^\alpha_R(\X;H_R)$ norm of $F$ is very restrictive. We do not know if it is possible to replicate the technique of \cite{Pri09} to remove this requirement. We intend to study this situation in a future paper.



\section{Example}\label{Sect_Examples}

Let $Q\in\mathcal{L}(\X)$ be a positive, self-adjoint and compact operator. For $\alpha, \beta\geq 0$  we set $A:=-(1/2)Q^{-\beta}:Q^{\beta}(\X)\subseteq\X\ra\X$ and $R:=Q^{\alpha}$. Let $\{e_k\}_{k \in\N}$ be an orthonormal basis of $\X$ consisting of eigenvectors of $Q$, and let $\{\lambda_k\}_{k\in\N}$ be the eigenvalues of $Q$ associated with $\{e_k\}_{k \in\N}$. Since $Q$ is a compact and positive operator, there exists $k_0\in\N$ such that  $0<\lambda_k\leq \lambda_{k_0}$, for any $k\in\N$. Without loss of generality we assume $k_0=1$. Hence, for any $x\in Q^{\beta}(\X)$, we have
\begin{equation}\label{disAX}
\scal{Ax}{x}=\sum_{k=1}^{+\infty}-\frac{1}{2}\lambda^{-\beta}_k \scal{x}{e_k}^2\leq -\frac{1}{2}\lambda^{-\beta}_1\norm{x}^2.
\end{equation}
Since $Q$ is a compact and positive operator, then $\{e_k\}_{k \in\N}\subseteq Q^{\beta}(\X)=\Dom(A)$ and $\Dom(A)$ is dense in $\X$, so $A$ generates a strongly continuous, analytic and contraction semigroup in $\X$. Let $A_{\alpha}$ be the part of $A$ in $H_\alpha:=H_{Q^\alpha}$, we recall that 
\[
\Dom(A_{\alpha}):=\{x\in Q^{\alpha}(\X)\cap Q^{\beta}(\X)\, |\, Ax\in Q^{\alpha}(\X)\}.
\]
By \eqref{disAX}, for any $x\in \Dom(A_{\alpha})$, we have
\begin{align*}
\scal{Ax}{x}_\alpha=\langle Q^{-\alpha}Ax, Q^{-\alpha}x\rangle =\langle AQ^{-\alpha}x,Q^{-\alpha}x\rangle \leq -\frac{1}{2}\lambda^{-\beta}_1\|Q^{-\alpha}x\|^2=-\frac{1}{2}\lambda^{-\beta}_1\norm{x}_\alpha^2.
\end{align*}
Since $Q^{\alpha+\beta}(\X)$ is dense in $\X$ and $Q^{-\alpha}$ is a closed operator in $\X$, then $Q^{\alpha+\beta}(\X)$ is dense in $H_\alpha$, moreover $Q^{\alpha+\beta}(\X)= \Dom(A_{\alpha})$. Hence $A$ generates a strongly continuous and contraction semigroup in $H_\alpha$. We refer to \cite[Section 5.4-5.5]{DA-ZA4} for a study of the validity of Hypothesis \ref{STR1}\eqref{STR1.2}.

So all the results of the paper can be applied in this case, in particular Theorem \ref{THM_SCHAUDER} and Theorem \ref{Thm_Zyg} hold true. Observe that
\[
Q_t=\int^t_0e^{2sA}Q^{2\alpha}ds=\frac{1}{2}Q^{2\alpha+\beta}(\Id-e^{-2tQ^{-\beta}}),
\]
so, arguing as in \cite[Example 1]{CL21}, for any $t>0$ it holds
\begin{align*}
\|Q_t^{-1/2}e^{tA}\|_{\mathcal{L}(\X)}\approx Ct^{-\alpha/\beta-1/2},
\end{align*}
for some positive constant $C$. So \eqref{doom} is verified only for $\alpha=0$. An example of a function $G:\X\ra\X$ is any radial function defined as 
\[G(x)=\Phi(\|x\|^2),\qquad x\in\X,\] 
where $\Phi:[0,+\infty)\ra\R$ is a three times differentiable function such that $\Phi'(r)=O(r^{-1})$ for $r\ra+\infty$.

\appendix

\section{A result about uniformly continuous functions}\label{appen}

We recall a result reguarding uniformly continuous functions that we have used throughout the paper. The proof is standard and follows the same ideas of the proof of \cite[Lemma 3.3]{Pri99}, we provide it for completeness. 

\begin{prop}\label{Zone400}
Let $\mathcal{M}$ be a separable metric space with metric $d_{\mathcal{M}}$, $(Y,\mu)$ be a measurable space ($\mu$ is a finite, positive and complete measure) and $Z$ be a Banach space with norm $\norm{\cdot}_Z$. Consider a function $F:\mathcal{M}\times Y\ra Z$ that satisfies
\begin{enumerate}[\rm (i)]
\item for any $m\in\mathcal{M}$, the map $y\mapsto F(m,y)$ is measurable;

\item for $\mu$-a.e. $y\in Y$, the map $m\mapsto F(m,y)$ is uniformly continuous;\label{Zone400,2}

\item there exists a $\mu$-integrable function $g:Y\ra\R$ such that for all $m\in\mathcal{M}$ and $\mu$-a.e. $y\in Y$ it holds $\|F(m,y)\|_Z\leq g(y)$.
\end{enumerate}
The map $h:\mathcal{M}\ra Z$, defined as
\begin{align*}
h(m):=\int_Y F(m,y) \mu(dy),\qquad m\in\mathcal{M},
\end{align*}
is bounded and uniformly continuous.
\end{prop}

\begin{proof}
The boundedness of $h$ is trivial and its continuity follows by the dominated convergence theorem (see \cite[Theorem 3, p. 45]{DU77}). So we just need to prove the uniform continuity of $h$. Let $N\subseteq Y$ be such that $\mu(N)=0$ and condition \eqref{Zone400,2} holds for all $y\in Y\setminus N$. For any $y\in Y\setminus N$, consider the modulus of continuity of the map $x\mapsto F(x,y)$ defined as
\begin{align*}
\omega_{F,y}(t):=\sup\set{\|F(m,y)-F(m',y)\|_Z\tc m,m'\in\mathcal{M},\ d_{\mathcal{M}}(m,m')\leq t},\qquad t>0.
\end{align*}
Observe that, by the separability of $\mathcal{M}\times\mathcal{M}$, for every $t>0$ there exists a countable set $D(t)\subseteq \mathcal{M}\times\mathcal{M}$ such that 
\begin{align*}
\omega_{F,y}(t):=\sup\set{\|F(m,y)-F(m',y)\|_Z\tc (m,m')\in D(t),\ d_{\mathcal{M}}(m,m')\leq t},\qquad t>0.
\end{align*}
The countability of $D(t)$ assures the measurability, with respect to $\mu$, of the map $y\mapsto\omega_{F,y}(t)$ for any $t>0$.
Moreover a standard computation gives that for any $y\in Y\setminus N$ and $t>0$ it holds $\omega_{F,y}(t)\leq 2g(y)$. Now let $(t_n)_{n\in\N}$ be a sequence of positive real numbers converging to zero. We have, for any $m,m'\in\X$ with $d_{\mathcal{M}}(m,m')\leq t_n$ 
\begin{align*}
\|h(m)-h(m')\|_Z\leq \int_Y\|F(m,y)-F(m',y)\|_Z\mu(dy)\leq \int_Y\omega_{F,y}(t_n)\mu(dy).
\end{align*}
The thesis follows by the dominated convergence theorem (see \cite[Theorem 3, p. 45]{DU77}).
\end{proof}

\section*{Declarations}


\subsection*{Fundings} The authors are members of GNAMPA (Gruppo Nazionale per l’Analisi Matematica, la Probabilit\`a
e le loro Applicazioni) of the Italian Istituto Nazionale di Alta Matematica (INdAM). The authors have been also partially supported by the research project PRIN 2015233N5A ``Deterministic and stochastic evolution equations'' of the Italian Ministry of Education, MIUR. S.F has been partially supported by the OK-INSAID project ARS01-00917. D.A.B. has been partially supported by the grant ``Stochastic differential equations and associated Markovian semigroups'' of the University of Pavia. The authors have no relevant financial or non-financial interests to disclose.

\subsection*{Research Data Policy and Data Availability Statements} Data sharing not applicable to this article as no datasets were generated or analysed during the current study.


\begin{thebibliography}{99}

\bibitem{AD-BA-MA1} 
Addona, D., Bandini, E.,  Masiero, F.
\newblock{\it A nonlinear Bismut-Elworthy formula for HJB equations with quadratic Hamiltonian in Banach spaces}, NoDEA Nonlinear Differential Equations Appl. {\bf 27}, Paper No. 37, 56 (2020).

\bibitem{AD-MA-PR1} 
Addona, D., Masiero, F., Priola E.
\newblock{\it A BSDEs approach to pathwise uniqueness for stochastic evolution equations}, eprint {\bf arXiv}:2110.01994 (2021).

\bibitem{ABF21} 
Angiuli, L., Bignamini, D.A., Ferrari, S.
\newblock{\it Logarithmic Harnack inequalities for transition semigroups in Hilbert spaces}, eprint {\bf arXiv}:2111.13250 (2021).

\bibitem{ABGP06} 
Athreya, S.R., Bass, R.F., Gordina, M., Perkins, E.A.
\newblock{\it Infinite dimensional stochastic differential equations of {O}rnstein--{U}hlenbeck type}, Stochastic Process. Appl. {\bf 116}, 381--406 (2006).

\bibitem{ABP05} 
Athreya, S.R., Bass, R.F., Perkins, E.A.
\newblock{\it H\"{o}lder norm estimates for elliptic operators on finite and infinite-dimensional spaces}, Trans. Amer. Math. Soc. {\bf 357}, 5001--5029 (2005).

\bibitem{Big21} 
Bignamini, D.A.
\newblock{\it $L^2$-theory for transitions semigroups associated to dissipative systems}, Stoch. PDE: Anal. Comp. (2022). \textbf{DOI}:10.1007/s40072-022-00253-x.

\bibitem{BF20} 
Bignamini, D.A., Ferrari, S.
\newblock{ \it Regularizing Properties of (Non-Gaussian) Transition Semigroups in Hilbert Spaces}, Potential Anal (2021). \textbf{DOI}:10.1007/s11118-021-09931-2.

\bibitem{BF22} 
Bignamini, D.A., Ferrari, S.
\newblock{\it On generators of transition semigroups associated to semilinear stochastic partial differential equations}, J. Math. Anal. Appl. {\bf 508}, Paper No. 125878, 40 (2022).

\bibitem{BON-FUR1} 
Bonaccorsi, S., Fuhrman, M.
\newblock{\it Regularity results for infinite dimensional diffusions. {A} {M}alliavin calculus approach}, Atti Accad. Naz. Lincei Cl. Sci. Fis. Mat. Natur. Rend. Lincei (9) Mat. Appl. {\bf 10}, 35--45 (1999).

\bibitem{CR88}
Bennett, C., Sharpley, R.
\newblock{\it Interpolation of operators}, Pure and Applied Mathematics, vol. 129, Academic Press, Inc., Boston, MA, 1988.

\bibitem{CDP96} 
Cannarsa, P., Da Prato, G.
\newblock{\it Infinite-dimensional elliptic equations with {H}\"{o}lder-continuous coefficients}, Adv. Differential Equations {\bf 1}, 425--452 (1996).

\bibitem{CDP96-2} 
Cannarsa, P., Da Prato, G.
\newblock{\it Schauder estimates for {K}olmogorov equations in {H}ilbert spaces}. In ``Progress in elliptic and parabolic partial differential equations ({C}apri, 1994)'', A. Alvino, P. Buonocore, V. Ferone, E. Giarrusso, S.
Matarasso, R. Toscano and G. Trombetti (editors), Pitman Res. Notes Math. Ser., vol. 350, Longman, Harlow, 100--111 (1996).

\bibitem{CER1}  
Cerrai, S.
\newblock{ \it Second order {PDE}'s in finite and infinite dimension}, Lecture Notes in Mathematics, vol. 1762, Springer-Verlag, Berlin, 2001.

\bibitem{CDP12} 
Cerrai, S., Da Prato, G.
\newblock{\it Schauder estimates for elliptic equations in {B}anach spaces associated with stochastic reaction-diffusion equations}, J. Evol. Equ. {\bf 12}, 83--98 (2012).

\bibitem{CL19} 
Cerrai, S., Lunardi, A.
\newblock{\it Schauder theorems for {O}rnstein--{U}hlenbeck equations in infinite dimension}, J. Differential Equations {\bf 267}, 7462--7482 (2019).

\bibitem{CL21} 
Cerrai, S., Lunardi, A.
\newblock{\it Smoothing effects and maximal H\"older regularity for non-autonomous Kolmogorov equations in infinite dimension}, eprint {\bf arXiv}:2111.05421 (2021).

\bibitem{DA6}
Da Prato, G. 
\newblock{\it A new regularity result for {O}rnstein-{U}hlenbeck generators and applications}, J. Evol. Equ. {\bf 3}, 485--498 (2003).

\bibitem{DA04}
Da Prato, G. 
\newblock{\it Kolmogorov equations for stochastic {PDE}s}, Advanced Courses in Mathematics. CRM Barcelona, Birkh\"{a}user Verlag, Basel, 2004.


\bibitem{DA5}
Da Prato, G.
\newblock{\it Schauder estimates for some perturbation of an infinite dimensional Ornstein--Uhlenbeck operator}, Discrete Contin. Dyn. Syst. Ser. S {\bf 6}, 637--647 (2013).

\bibitem{DFPR13} 
Da Prato, G., Flandoli, F., Priola, E., R\"{o}ckner, M.
\newblock{\it Strong uniqueness for stochastic evolution equations in {H}ilbert spaces perturbed by a bounded measurable drift}, Ann. Probab. {\bf 41}, 35--45 (2013).

\bibitem{DL95} 
Da Prato, G., Lunardi, A.
\newblock{\it On the {O}rnstein-{U}hlenbeck operator in spaces of continuous functions}, J. Funct. Anal. {\bf 131}, 94--114 (1995).

\bibitem{DPZ02}
Da Prato, G., Zabczyk, J.
\newblock{ \it Second order partial differential equations in {H}ilbert spaces}, London Mathematical Society Lecture Note Series, vol. 293, Cambridge University Press, Cambridge, 2002.

\bibitem{DA-ZA4}
Da Prato, G., Zabczyk, J.
\newblock{ \it Stochastic equations in infinite dimensions}, Encyclopedia of Mathematics and its Applications, vol. 152, Cambridge University Press, Cambridge, 2014.

\bibitem{DU77}
Diestel, J., Uhl, Jr., J.J.
\newblock{\it Vector measures}, Mathematical Surveys, No. 15, American Mathematical Society, Providence, R.I., 1977.

\bibitem{EL-LI1}
Elworthy, K.D., Li, X.-M.
\newblock{\it Formulae for the derivatives of heat semigroups}, J. Funct. Anal. {\bf 125}, 252--286 (1994).

\bibitem{EN00}
Engel, K.-J., Nagel, R.
\newblock{ \it One-parameter semigroups for linear evolution equations}, Graduate Texts in Mathematics, vol. 194, Springer-Verlag, New York, 2000.

\bibitem{FHHMZ11}
Fabian, M., Habala, P., H\'{a}jek, P., Montesinos, V., Zizler, V.
\newblock{\it Banach space theory}, CMS Books in Mathematics/Ouvrages de Math\'{e}matiques de la SMC, Springer, New York, 2011.

\bibitem{FUR1}  
Fuhrman, M.
\newblock{\it Smoothing properties of nonlinear stochastic equations in {H}ilbert spaces}, NoDEA Nonlinear Differential Equations Appl. {\bf 3}, 445--464 (1996).

\bibitem{GT01}
Gilbarg, D., Trudinger, N.S.
\newblock{ \it Elliptic partial differential equations of second order}, Classics in Mathematics, Springer-Verlag, Berlin, 2001.

\bibitem{GK01}
Goldys, B., Kocan, M.
\newblock{\it Diffusion Semigroups in Spaces of Continuous Functions with Mixed Topology}, J. Differential Equations {\bf 173}, 17--39 (2001).

\bibitem{GRO1}  
Gross, L.
\newblock{ \it Potential theory on Hilbert space}, J. Funct. Anal. {\textbf{1}}, 123--181 (1967).


\bibitem{KE1} 
Kechris, A.S.
\newblock{ \it Classical descriptive set theory}, Graduate Texts in Mathematics, vol. 156, Springer-Verlag, New York, 1995.

\bibitem{KCL75}  
Kru\v{z}kov, S.N., Castro, A., Lopes, M.
\newblock{ \it Schauder type estimates, and theorems on the existence of the solution of fundamental problems for linear and nonlinear parabolic equations}, Dokl. Akad. Nauk SSSR {\textbf{220}} (Russian), 277--280 (1975). English transl.: Soviet
Math. Dokl. {\bf 16}, 60--64 (1975).

\bibitem{KCL80}  
Kru\v{z}kov, S.N., Castro, A., Lopes, M.
\newblock{ \it Mayoraciones de Schauder y teorema de existencia de las soluciones del problema de Cauchy para ecuaciones parabolicas lineales y no lineales}, (I) Ciencias Matem\'aticas {\bf 1}, 55--76 (1980); (II) Ciencias Matem\'aticas {\bf 3}, 37--56 (1982).


\bibitem{KZ21} 
Kry\v stof, V., Zaj\'\i\v cek, L.
\newblock{\it Functions on a convex set which are both  $\omega$-semiconvex and  $\omega$-semiconcave}, J. Convex Anal. {\textbf{29}}, (2022).

\bibitem{KUO1}  
Kuo, H.H.
\newblock{ \it Gaussian measures in {B}anach spaces}, Lecture Notes in Mathematics, vol. 463, Springer-Verlag, Berlin-New York, 1975.

\bibitem{LSU68}
Lady\v{z}enskaja, O.A., Solonnikov, V.A., Ural'ceva, N.N.
\newblock{ \it Linear and quasilinear equations of parabolic type}, Translations of Mathematical Monographs, vol. 23, American Mathematical Society, Providence, R.I., 1968.

\bibitem{LU68}
Lady\v{z}enskaja, O.A., Ural'ceva, N.N.
\newblock{ \it Linear and quasilinear elliptic equations}, Academic Press, New York-London, 1968.

\bibitem{LL86} 
Lasry, J.-M., Lions, P.-L.
\newblock{\it A remark on regularization in {H}ilbert spaces}, Israel J. Math. {\bf 55}, 257--266 (1986).

\bibitem{LI-RO1}
Liu, W., R\"{o}ckner, M.
\newblock{ \it Stochastic partial differential equations: an introduction}, Universitext, Springer, Cham, 2015.

\bibitem{Lor17}
Lorenzi, L.
\newblock{ \it Analytical methods for {K}olmogorov equations}, Monographs and Research Notes in Mathematics, CRC Press, Boca Raton, FL, 2017.


\bibitem{Lun97} 
Lunardi, A.
\newblock{\it Schauder estimates for a class of degenerate elliptic and parabolic operators with unbounded coefficients in {${\bf R}^n$}}, Ann. Scuola Norm. Sup. Pisa Cl. Sci. (4) {\bf 24}, 133--164 (1997).

\bibitem{Lun18}
Lunardi, A.
\newblock{ \it Interpolation theory}, Appunti. Scuola Normale Superiore di Pisa (Nuova Serie) [Lecture Notes. Scuola Normale Superiore di Pisa (New Series)], vol. 16, Third Edition, Edizioni della Normale, Pisa, 2018.

\bibitem{LR21} 
Lunardi, A., R\"{o}ckner, M.
\newblock{\it Schauder theorems for a class of (pseudo-)differential operators on finite- and infinite-dimensional state spaces}, J. Lond. Math. Soc. (2) {\bf 104}, 492--540 (2021).

\bibitem{MAS2} 
Masiero, F.
\newblock{\it Semilinear {K}olmogorov equations and applications to stochastic optimal control}, Appl. Math. Optim. {\bf 51}, 201--250 (2005).

\bibitem{MAS1} 
Masiero, F.
\newblock{\it Regularizing properties for transition semigroups and semilinear parabolic equations in {B}anach spaces}, Electron. J. Probab. {\bf 12}, no. 13, 387--419 (2007).

\bibitem{Mas08} 
Masiero, F.
\newblock{\it Stochastic optimal control problems and parabolic equations in {B}anach spaces}, SIAM J. Control Optim. {\bf 47}, 251--300 (2008).

\bibitem{Oks03}
\O ksendal, B.
\newblock{\it Stochastic differential equations}, Universitext, Springer-Verlag, Berlin, 2003.

\bibitem{PES-ZA1} 
Peszat, S., Zabczyk, J.
\newblock{\it Strong {F}eller property and irreducibility for diffusions on {H}ilbert spaces}, Ann. Probab. {\bf 23}, 157--172 (1995).


\bibitem{PRI1}
Priola, E.
\newblock{ \it Partial differential equations with infinitely many variables}, Iris, AperTO, Universit\`a degli Studi di Torino, 1999.

\bibitem{Pri99} 
Priola, E.
\newblock{\it On a class of {M}arkov type semigroups in spaces of uniformly continuous and bounded functions}, Studia Math. {\bf 136}, 271--295 (1999).

\bibitem{Pri09} 
Priola, E.
\newblock{\it Global Schauder estimates for a class of degenerate Kolmogorov equations}, Studia Math. {\bf 194}, 117--153 (2009).

\bibitem{PZ00} 
Priola, E., Zambotti, L.
\newblock{\it New optimal regularity results for infinite-dimensional elliptic equations}, Boll. Unione Mat. Ital. Sez. B Artic. Ric. Mat. (8) {\bf 3}, 411--429 (2000).

\bibitem{RE-SI1}
Reed, M., Simon, B.
\newblock{ \it Methods of modern mathematical physics. I. Functional analysis}, Academic Press, New York-London, 1972.

\bibitem{Rud76}
Rudin, W.
\newblock{ \it Principles of mathematical analysis}, International Series in Pure and Applied Mathematics, Third Edition, McGraw-Hill Book Co., New York-Auckland-D\"{u}sseldorf, 1976.


\bibitem{Tri95}
Triebel, H.
\newblock{ \it Interpolation theory, function spaces, differential operators}, Second Edition, Johann Ambrosius Barth, Heidelberg, 1995.

\bibitem{Wil91}
Williams, D.
\newblock{\it Probability with martingales}, Cambridge Mathematical Textbooks, Cambridge University Press, Cambridge, 1991.

\bibitem{Zam99} 
Zambotti, L.
\newblock{\it Infinite-dimensional elliptic and stochastic equations with {H}\"{o}lder-continuous coefficients}, Stochastic Anal. Appl. {\bf 17}, 487--508 (1999).


\end{thebibliography}
\end{document}